\documentclass[11pt,oneside,english]{amsart}
\usepackage[latin9]{inputenc}
\usepackage[a4paper]{geometry}
\usepackage{babel}
\usepackage{slashed}
\usepackage{amstext}
\usepackage{amsthm}
\usepackage{amssymb}
\usepackage{stackrel}
\usepackage{setspace}
\usepackage{color}
\usepackage[unicode=true,pdfusetitle,
 bookmarks=true,bookmarksnumbered=false,bookmarksopen=false,
 breaklinks=false,pdfborder={0 0 1},backref=false,colorlinks=false]
 {hyperref}

\makeatletter
\numberwithin{equation}{section}
\numberwithin{figure}{section}
\theoremstyle{plain}
\newtheorem{thm}{\protect\theoremname}[section]
\theoremstyle{plain}
\newtheorem{question}[thm]{\protect\questionname}
\theoremstyle{definition}
\newtheorem{defn}[thm]{\protect\definitionname}
\theoremstyle{remark}
\newtheorem*{acknowledgement*}{\protect\acknowledgementname}
\theoremstyle{plain}
\newtheorem{prop}[thm]{\protect\propositionname}
\theoremstyle{remark}
\newtheorem{rem}[thm]{\protect\remarkname}
\theoremstyle{plain}
\newtheorem{lem}[thm]{\protect\lemmaname}
\theoremstyle{plain}
\newtheorem{cor}[thm]{\protect\corollaryname}

\usepackage{babel}
\newtheorem{claim}{Claim}

\newtheorem{theorem}{Theorem}
\usepackage[shortlabels]{enumitem}

\providecommand{\acknowledgementname}{Acknowledgement}

\providecommand{\corollaryname}{Corollary}
\providecommand{\definitionname}{Definition}
\providecommand{\lemmaname}{Lemma}
\providecommand{\propositionname}{Proposition}
\providecommand{\questionname}{Question}
\providecommand{\remarkname}{Remark}
\providecommand{\theoremname}{Theorem}

\makeatother

\providecommand{\acknowledgementname}{Acknowledgement}
\providecommand{\corollaryname}{Corollary}
\providecommand{\definitionname}{Definition}
\providecommand{\lemmaname}{Lemma}
\providecommand{\propositionname}{Proposition}
\providecommand{\questionname}{Question}
\providecommand{\remarkname}{Remark}
\providecommand{\theoremname}{Theorem}

\begin{document}
\title[A number theoretic characterization of $E$-smooth and (FRS) morphisms]{A number theoretic characterization of $E$-smooth and (FRS) morphisms:
estimates on the number of $\mathbb{Z}/p^{k}\mathbb{Z}$-points}
\author{Raf Cluckers}
\address{Univ.~Lille, CNRS, UMR 8524 - Laboratoire Paul Painlev\'e, F-59000
Lille, France, and KU Leuven, Department of Mathematics, B-3001 Leuven,
Belgium}
\email{Raf.Cluckers@univ-lille.fr}
\urladdr{http://rcluckers.perso.math.cnrs.fr/}
\author{Itay Glazer}
\address{Department of Mathematics, Northwestern University, 2033 Sheridan
Road, Evanston, IL 60208, USA}
\email{itayglazer@gmail.com}
\urladdr{https://sites.google.com/view/itay-glazer}
\author{Yotam I. Hendel}
\address{Department of Mathematics, Northwestern University, 2033 Sheridan
Road, Evanston, IL 60208, USA}
\email{yotam.hendel@gmail.com}
\urladdr{https://sites.google.com/view/yotam-hendel}
\begin{abstract}
We provide uniform estimates on the number of $\mathbb{Z}/p^{k}\mathbb{Z}$-points
lying on fibers of flat morphisms between smooth varieties whose fibers
have rational singularities, termed (FRS) morphisms. For each individual
fiber, the estimates were known by work of Avni and Aizenbud, but
we render them uniform over all fibers. The proof technique for individual
fibers is based on Hironaka's resolution of singularities and Denef's
formula, but breaks down in the uniform case. Instead, we use recent
results from the theory of motivic integration. Our estimates are
moreover equivalent to the (FRS) property, just like in the absolute
case by Avni and Aizenbud. In addition, we define new classes of morphisms,
called $E$-smooth morphisms ($E\in\mathbb{N}$),  which refine the
(FRS) property, 
 and use the methods we developed to provide uniform
number-theoretic estimates as above for their fibers.  Similar estimates are given for fibers of $\varepsilon$-jet flat morphisms, improving previous results by the last two authors. 
\end{abstract}

\maketitle
\global\long\def\nats{\mathbb{\mathbb{N}}}%
\global\long\def\reals{\mathbb{\mathbb{R}}}%
\global\long\def\ints{\mathbb{\mathbb{Z}}}%
\global\long\def\val{\mathbb{\mathrm{val}}}%
\global\long\def\Qp{\mathbb{Q}_{p}}%
\global\long\def\Zp{\mathbb{Z}_{p}}%
\global\long\def\ac{\mathrm{ac}}%
\global\long\def\complex{\mathbb{C}}%
\global\long\def\rats{\mathbb{Q}}%
\global\long\def\supp{\mathrm{supp}}%
\global\long\def\VF{\mathrm{VF}}%
\global\long\def\RF{\mathrm{RF}}%
\global\long\def\VG{\mathrm{VG}}%
\global\long\def\spec{\mathrm{Spec}}%
\global\long\def\Ldp{\mathcal{L}_{\mathrm{DP}}}%
\raggedbottom

\global\long\def\llp{\mathopen{(\!(}}%
\global\long\def\llb{\mathopen{[[}}%
\global\long\def\rrp{\mathopen{)\!)}}%
\global\long\def\rrb{\mathopen{]]}}%

\section{Introduction}

\subsection{Overview}

Let $\varphi:X\to Y$ be an algebraic morphism between smooth $K$-varieties,
where $K$ is a number field. In this paper we give uniform arithmetic
and analytic equivalent characterizations to the (FRS) property of
$\varphi$, namely to the property of being flat with reduced fibers
of rational singularities (see Theorem \ref{Thm A}). These results
can be viewed as a common uniform improvement of the following two
theorems: 
\begin{enumerate}
\item \cite[Theorem A]{AA18}, where bounds were given on the number of
$\ints/p^{k}\ints$-points of reduced local complete intersection
schemes which have rational singularities (see also Theorem \ref{thm: counting points absolute}). 
\item \cite[Theorem 3.4]{AA16}, where pushforward of smooth measures with
respect to $\varphi$ over non-Archimedean local fields were shown
to have bounded density if and only if $\varphi$ is an (FRS) morphism
(see also Theorem \ref{thm:analytic criterion of the (FRS) property}). 
\end{enumerate}
In order to prove our uniform characterizations of the (FRS) property,
it seems natural to try and adapt the algebro-geometric proof of \cite[Theorem A]{AA18}
to the relative case. This fails to work because of unsatisfactory
behavior of resolution of singularities in families, with respect
to taking points over $\ints$, $\ints/p^{k}\ints$ and $\ints_{p}$
(see Section \ref{sec: difficulties- thm A}). Instead, we prove a
model theoretic result of independent interest about approximating
suprema of a certain sub-class of motivic functions, which we call
formally non-negative functions (see Theorem \ref{Thm B}). Using
Theorem \ref{Thm B} and by analyzing the jets of $\varphi$, we prove
Theorem \ref{Thm A}. Theorem \ref{Thm B} further strengthens \cite[Theorem 2.1.3]{CGH18}
in the case of formally non-negative functions. 
Finally, we provide uniform
estimates on the number of $\ints/p^{k}\ints$-points lying on 
fibers of $E$-smooth morphisms, 
 a new notion we introduce which refines the (FRS) property 
 ($E \in \nats$). 
 Uniform estimates are also provided for fibers of $\varepsilon$-jet
flat morphisms, achieving optimal bounds (c.f.~\cite[Theorem 8.18]{GHb}).
See Section \ref{subsec:E-smooth-and--jet flat} and Theorems
\ref{thm:characterization of E-smooth morphisms} and \ref{thm:characterization of epsilon-jet flatness} for these notions and results.

\subsection{Counting points over $\protect\ints/p^{k}\protect\ints$ : the absolute
case}

Let $X$ be a finite type $\ints$-scheme. The study of the quantity
$\#X(\ints/n\ints)$, and its asymptotic behavior in $n\in\nats$,
is a long standing problem in number theory. When $n=p$ is prime,
the asymptotic behavior is understood by the Lang-Weil estimates \cite{LW54},
and in particular, the family 
\[
\left\{ \frac{\#X(\ints/p\ints)}{p^{\mathrm{dim}X_{\rats}}}\right\} _{p}
\]
is uniformly bounded.

Moving to the case where $n=p^{k}$ is a prime power (which suffices,
by the Chinese remainder theorem), one can observe the following;
if $X$ is smooth as a $\ints$-scheme, then an application of Hensel's
lemma shows that 
\[
\left\{ \frac{\#X(\ints/p^{k}\ints)}{p^{k\mathrm{dim}X_{\rats}}}\right\} _{p,k}
\]
is uniformly bounded in both $p$ and $k$. On the other hand, taking
the non-reduced scheme $X=\mathrm{Spec}\mathbb{Z}[x]/(x^{2})$, we
see that 
\[
\frac{\#X(\ints/p^{2k}\ints)}{p^{2k\mathrm{dim}X_{\rats}}}=\#X(\ints/p^{2k}\ints)=p^{k},
\]
which is not uniformly bounded. The following natural question arises. 
\begin{question}
\label{que:absolute case}Is there a necessary and sufficient condition
on $X$ such that $\left\{ \frac{\#X(\ints/p^{k}\ints)}{p^{k\mathrm{dim}X_{\rats}}}\right\} _{p,k}$
is uniformly bounded?
\end{question}

In \cite{AA18}, Aizenbud and Avni, relying on results of Musta\c{t}\u{a}
\cite{Mus01} and Denef \cite{Den87}, gave such a necessary and sufficient
condition in the case where $X_{\rats}$ is a local complete intersection. 
\begin{defn}
\label{def:rational singularities}Let $K$ be a field of characteristic
$0$. A $K$-scheme of finite type $X$ has \textit{rational singularities}
if it is normal and for every resolution of singularities $\pi:\widetilde{X}\rightarrow X$,
one has $R^{i}\pi_{*}(O_{\widetilde{X}})=0$ for $i\geq1$. 
\end{defn}

\begin{thm}[{{see \cite[Theorem A]{AA18} and \cite{Gla19}}}]
\label{thm: counting points absolute}Let $X$ be a finite type $\mathbb{Z}$-scheme
such that $X_{\mathbb{Q}}$ is equidimensional and a local complete
intersection. Then the following are equivalent: 
\begin{enumerate}
\item $X_{\mathbb{Q}}$ has rational singularities (and, in particular,
$X_{\mathbb{Q}}$ is reduced). 
\item There exists $C>0$ such that for every prime $p$ and every $k\in\nats$
one has 
\[
\frac{\#X(\ints/p^{k}\ints)}{p^{k\mathrm{dim}X_{\rats}}}<C.
\]
\item There exists $C>0$ such that for every prime $p$ and every $k\in\nats$
one has 
\[
\left|\frac{\#X(\ints/p^{k}\ints)}{p^{k\mathrm{dim}X_{\rats}}}-\frac{\#X(\ints/p\ints)}{p^{\mathrm{dim}X_{\rats}}}\right|<Cp^{-1}.
\]
\end{enumerate}
\end{thm}

Motivated 

\subsection{Counting points over $\protect\ints/p^{k}\protect\ints$ : the relative
case}

Let $X$ and $Y$ be smooth finite type $\ints$-schemes and let $\varphi:X\rightarrow Y$
be a dominant morphism. Our goal in this paper is to treat the relative
analogue of Question \ref{que:absolute case}: 
\begin{question}
\label{que:relative case}Is there a necessary and sufficient condition
on $\varphi$ such that the size of each fiber of $\varphi:X(\ints/p^{k}\ints)\rightarrow Y(\ints/p^{k}\ints)$,
normalized by $p^{k(\mathrm{dim}X_{\rats}-\mathrm{dim}Y_{\rats})}$,
is uniformly bounded when varying $p,k$ and $y\in Y(\ints/p^{k}\ints)$? 
\end{question}

Since the Lang-Weil estimates are effective uniformly over all schemes
of bounded complexity, Question \ref{que:relative case} is easily
answered in the case where $k=1$; the condition that \textbf{$\varphi_{\rats}$
is flat} is necessary and sufficient (see \cite[Theorem 8.4]{GHb}).
For the general case, we use the following notion from \cite[Definition II]{AA16}.
By a $K$-variety with $K$ a field we mean a reduced $K$-scheme
of finite type. 
\begin{defn}
\label{def:(FRS)}Let $X$ and $Y$ be smooth $K$-varieties, where
$K$ is a field with $\mathrm{char}(K)=0$. We say that a morphism
$\varphi:X\rightarrow Y$ is \textit{(FRS)} if it is flat and if every
fiber of $\varphi$ has rational singularities. 
\end{defn}

\subsection{Main results}

We are now ready to state the main result of this paper.

\begin{theorem}[See Theorem \ref{thm:number theoretic characterization of the (FRS) property} for a more general version]\label{Thm A}Let
$\varphi:X\to Y$ be a dominant morphism between finite type $\ints$-schemes
$X$ and $Y$, with $X_{\rats},Y_{\rats}$ smooth and geometrically
irreducible. Then the following are equivalent: 
\begin{enumerate}
\item $\varphi_{\rats}:X_{\rats}\to Y_{\rats}$ is (FRS). 
\item There exists $C_{1}>0$ such that for every prime $p$, $k\in\nats$
and $y\in Y(\ints/p^{k}\ints)$ one has 
\[
\frac{\#\varphi^{-1}(y)}{p^{k(\mathrm{dim}X_{\rats}-\mathrm{dim}Y_{\rats})}}<C_{1}.
\]
\item There exists $C_{2}>0$ such that for every prime $p$, $k\in\nats$
and $y\in Y(\ints/p^{k}\ints)$ one has 
\[
\left|\frac{\#\varphi^{-1}(y)}{p^{k(\mathrm{dim}X_{\rats}-\mathrm{dim}Y_{\rats})}}-\frac{\#\varphi^{-1}(\bar{y})}{p^{(\mathrm{dim}X_{\rats}-\mathrm{dim}Y_{\rats})}}\right|<C_{2}p^{-1},
\]
where $\overline{y}$ is the image of $y$ under the reduction $Y(\ints/p^{k}\ints)\rightarrow Y(\mathbb{F}_{p})$. 
\item There exists $C_{3}>0$ such that the following hold for every prime
$p$. Let $\mu_{X(\Zp)}$ and $\mu_{Y(\Zp)}$ be the canonical measures
on $X(\Zp)$ and $Y(\Zp)$ (see Lemma-Definition \ref{lem:canonical measure}).
Then the pushforward measure $\varphi_{*}\mu_{X(\Zp)}$ has continuous
density $f_{p}$ with respect to $\mu_{Y(\Zp)}$, and $\left\Vert f_{p}\right\Vert _{\infty}<C_{3}$. 
\end{enumerate}
\end{theorem}

Using a jet-scheme characterization of rational singularities by Musta\c{t}\u{a}
\cite{Mus01,Mus02}, it can be shown that a morphism $\varphi:X\rightarrow Y$
between smooth schemes is (FRS) if and only if for each $k\in\nats$,
every non-empty fiber of the corresponding $k$-th jet map $J_{k}(\varphi):J_{k}(X)\rightarrow J_{k}(Y)$
is of dimension $\mathrm{dim}J_{k}(X)-\mathrm{dim}J_{k}(Y)$ (i.e.
$J_{k}(\varphi)$ is flat) and has a singular locus of codimension
at least $1$ (see Subsection \ref{subsec:E-smooth-and--jet flat}
and Lemma \ref{lem:1-smooth is (FRS)}). Based on this characterization,
it is natural to define two variations of the (FRS) property. 
\begin{itemize}
\item A morphism $\varphi$ is \emph{$\varepsilon$-jet flat}, for $\varepsilon\in\reals_{>0}$,
if the fibers of $J_{k}(\varphi)$ are of dimension at most $\mathrm{dim}J_{k}(X)-\varepsilon\mathrm{dim}J_{k}(Y)$,
for all $k\in\nats$ (see \cite[Definition 3.22]{GHb}).
\item A morphism $\varphi$ is called \emph{$E$-smooth} if it is $1$-jet
flat, and each of the fibers of $J_{k}(\varphi)$ has singular locus
of codimension at least $E$.
\end{itemize}
In Section \ref{subsec:Number-theoretic-estimates-for}, using methods
similar to the proof of Theorem \ref{Thm A}, we provide uniform estimates
on the fibers of $E$-smooth and $\varepsilon$-jet flat morphisms
(see Theorems \ref{thm:characterization of E-smooth morphisms} and
\ref{thm:characterization of epsilon-jet flatness}). 
In particular, uniform estimates are given on fibers of flat morphisms whose fibers have terminal or log-canonical singularities.

\subsubsection{\label{sec: difficulties- thm A}Main difficulties in the proof of
Theorem \ref{Thm A}}

The proof of Theorem \ref{thm: counting points absolute} in \cite{AA18}
proceeds by (locally) embedding $X$ as a complete intersection in
$\mathbb{A}^{N}$ and choosing an embedded resolution of singularities
for the pair $(X_{\mathbb{Q}},\mathbb{A}_{\rats}^{N})$, also called
a log-resolution, whose existence follows from \cite{Hir:Res}. For
large $p$, one can then use Denef's formula \cite[Theorem 3.1]{Den87},
to relate $\#X(\ints/p^{k}\ints)$ to $\left\{ \#E_{I}(\mathbb{F}_{p})\right\} _{I}$
and numerical data associated to the choice of resolution, where $\{E_{I}\}_{I}$
is a collection of constructible subsets built out of the prime divisors
$\{E_{i}\}_{i=1}^{M}$ appearing in such a resolution. Combined with
the Lang-Weil estimates for the $E_{I}$'s, this yields estimates
for $\#X(\ints/p^{k}\ints)$. To finally achieve the bounds of Theorem
\ref{thm: counting points absolute}, one needs the reductions modulo
$p$ of the $E_{I}$'s to be of the expected dimensions over $\mathbb{F}_{p}$.
This can always be done if the prime $p$ is large enough (small primes
are treated separately in \cite{Gla19}).

\medskip{}

If $\varphi_{\rats}:X_{\rats}\rightarrow Y_{\rats}$ is (FRS), its
fibers are local complete intersections with rational singularities,
and one may try to mimic the strategy for Theorem \ref{thm: counting points absolute}.
The weak point is that this only seems to work for each fiber separately,
but does not give the desired uniformity in the choice of fiber. One
can try to make this naive fiber-wise strategy more uniform by choosing
some simultaneous resolutions of singularities. This can be done by
breaking $Y$ into constructible subsets, with resolutions over generic
points of the pieces. However, such finite partition of $Y$ into
constructible sets does not behave well at all with respect to taking
points over the rings $\ints$, $\mathbb{Z}/p^{k}\mathbb{Z}$, or
$\Zp$. In fact, as far as we can see, the approach with resolutions
of singularities in families is hard to adapt to the family situation
of Theorem \ref{Thm A}.

To avoid these difficulties, we use the motivic nature of $\ints/p^{k}\ints$-point
count of the fibers of $\varphi$, that is, we use insights from motivic
integration and uniform $p$-adic integration. Let $r_{k}:Y(\Zp)\rightarrow Y(\ints/p^{k}\ints)$
be the reduction map. Write $d:=\mathrm{dim}X_{\mathbb{Q}}-\mathrm{dim}Y_{\mathbb{Q}}$.
For each prime $p$, each $y\in Y(\Zp)$ and each integer $k\ge1$
we set 
\begin{equation}
g_{p}(y,k)=\frac{\#\varphi^{-1}(r_{k}(y))}{p^{kd}}\text{~ and ~}\widetilde{h}_{p}(y,k):=g_{p}(y,k)-g_{p}(y,1),\label{eq: definition of g and h}
\end{equation}
as in the left-hand side of Items (2) and (3) of Theorem \ref{Thm A}.
The collections of functions $\{g_{p}\}_{p},\{\widetilde{h}_{p}\}_{p}$
are examples of \emph{motivic functions}, namely in a uniform $p$-adic
sense as in \cite{CGH18}, but closely related to genuine motivic
constructible functions from \cite{CL08}. We use motivic integration
to extract information on $\{g_{p}\}_{p}$ and $\{\widetilde{h}_{p}\}_{p}$,
which in turn allows us to prove Theorem \ref{Thm A}. 

The proofs of the number-theoretic estimates for $\varepsilon$-jet
flat and $E$-smooth morphisms (Theorems \ref{thm:characterization of E-smooth morphisms}
and \ref{thm:characterization of epsilon-jet flatness}) share similar
difficulties with the proof of Theorem \ref{Thm A}. Theorem \ref{thm:characterization of epsilon-jet flatness}
improves previous bounds for $\varepsilon$-jet flat morphisms: 
the bounds given in \cite[Corollary 2.9]{VZG08}
 on $g_{p}(y,k)$ are uniform in $k$, but not in
$p$ and $y$ (see
Remark \ref{rem:relating jet flatness} for the relation of $\varepsilon$-jet
flatness to the log canonical threshold), and the bounds given in \cite[Theorem 8.18]{GHb} are
uniform in $p,y,k$, but are not optimal.

\subsubsection{Model-theoretic results}

We denote by $\mathrm{Loc}$ the collection of all non-Archimedean
local fields, by $\mathrm{Loc}_{0}$ the collection of all $F\in\mathrm{Loc}$
of characteristic zero, and by $\mathrm{Loc}_{\gg}$ the collection
of all $F\in\mathrm{Loc}$ with large enough residual characteristic,
where 'large enough' changes according to our needs.

Let $\Ldp$ denote the Denef-Pas language. This is a first order language
with three sorts to account for a valued field $F$, a residue field
$k_{F}$ and a value group which we identify with $\ints$. An $\Ldp$-definable
set $X=\{X_{F}\}_{F\in\mathrm{Loc}_{\gg}}$ is a collection of subsets
$X_{F}\subseteq F^{n_{1}}\times k_{F}^{n_{2}}\times\ints^{n_{3}}$
which is uniformly defined using an $\Ldp$-formula. Given $\Ldp$-definable
sets $X$ and $Y$, a collection of functions $\{f:X_{F}\to Y_{F}\}_{F\in\mathrm{Loc}_{\gg}}$
is called an ($\Ldp$-)definable function if its graph is definable.

Given a definable set $X=\{X_{F}\}_{F\in\mathrm{Loc}_{\gg}}$, the
ring of motivic functions $\mathcal{C}(X)$ is a certain natural class
of functions whose building blocks are the definable functions, and
is closed under integration. Built on a natural notion of positivity,
we define the semi-ring of formally non-negative functions $\mathcal{C}_{+}(X)\subset\mathcal{C}(X)$.
As an example, the collection $\{\varphi_{*}\mu_{F}\}_{F\in\mathrm{Loc}_{\gg}}$
of pushforwards of Haar measures $\mu_{F}$ on $\mathcal{O}_{F}^{n}$
under any polynomial map $\varphi$, as well as $\{g_{p}\}_{p}$ above
are formally non-negative motivic functions. The classes $\mathcal{C}_{+}(X)$
and $\mathcal{C}(X)$ above are uniform $p$-adic specializations
of more genuinely motivic functions defined in \cite{CL08,CL10},
but they go by similar methods and theories. See Subsection \ref{subsec:Motivic-functions}
for further details on motivic functions.

As a key step towards proving Theorem \ref{Thm A}, we show the following
strengthening of \cite[Theorem 2.1.3]{CGH18} for the class of formally
non-negative motivic functions:

\begin{theorem}[Theorem \ref{thm:improved supremum theorem}]\label{Thm B}

Let $f$ be in $\mathcal{C}_{+}(X\times W)$, where $X$ and $W$
are $\Ldp$-definable sets. Then there exists a constant $C>0$, and
a function $G\in\mathcal{C}_{+}(X)$ such that for any $F\in\mathrm{Loc}_{\gg}$
and any $x\in X_{F}$ such that $w\mapsto f_{F}(x,w)$ is bounded
on $W_{F}$, we have 
\[
\underset{w\in W_{F}}{\sup}f_{F}(x,w)\leq G_{F}(x)\leq C\cdot\underset{w\in W_{F}}{\sup}f_{F}(x,w).
\]

\end{theorem}

The approximation of suprema given in Theorem \ref{Thm B} is best
possible for the class of formally non-negative motivic functions
$\mathcal{C}_{+}(X\times W)$, in the sense that one cannot choose
$C$ to be a universal constant (see Proposition \ref{prop:approximation is tight}).
In \cite[Theorem 2.1.3]{CGH18}, a similar approximation result is
shown (for motivic functions in $\mathcal{C}(X\times W)$ and in $\mathcal{C}^{\mathrm{exp}}(X\times W)$),
but where the constant $C$ is replaced by $q_{F}^{C}$, with $q_{F}$
the number of elements in the residue field $k_{F}$ of $F$, and
where instead of $\sup f_{F}$ one approximates $\sup|f_{F}|^{2}$.
For more details on the optimality of these approximation results,
see the discussion in Subsection \ref{subsec:Optimality-of-the bounds}.

\subsubsection{Sketch of proof of Theorem \ref{Thm A}}

To prove Theorem \ref{Thm A}, we show $(1)\Rightarrow(3)\Rightarrow(2)\Rightarrow(4)\Rightarrow(1)$.
The implications $(3)\Rightarrow(2)\Rightarrow(4)$ are rather easy
and the implication $(4)\Rightarrow(1)$ essentially follows from
an equivalent analytic characterization of the (FRS) property due
to Aizenbud-Avni (see Theorem \ref{thm:analytic criterion of the (FRS) property}).
The challenging part of the proof is to show $(1)\Rightarrow(3)$.
Small primes are dealt using Theorem \ref{thm:analytic criterion of the (FRS) property}
and using basic properties of the canonical measure (Lemma-Definition
\ref{lem:canonical measure}). Thus we may consider only large enough
primes $p$. Let us sketch the main strategy of the proof of $(1)\Rightarrow(2)$,
for large $p$, which has similar difficulties to $(1)\Rightarrow(3)$. 
\begin{enumerate}
\item[(a)] We use Theorem \ref{thm:analytic criterion of the (FRS) property}
to show that 
\[
\underset{y,k}{\sup}\,g_{p}(y,k)<C(p)
\]
for some constant $C(p)$ depending on $p$. 
\item[(b)] Item (a) and the fact that $g$ is formally non-negative as a motivic
function (see Definition \ref{def:Presburger constructible functions}),
allow us to utilize Theorem \ref{Thm B} to approximate (for each
$k$ and $p$) 
\[
\underset{y\in Y(\Zp)}{\sup}g_{p}(y,k)
\]
by $G_{p}(k)$ for a single motivic function $\{G_{p}:\ints_{\geq1}\rightarrow\reals\}_{p}$. 
\item[(c)] We use results from \cite{CGH18} on approximate suprema of constructible
Presburger functions to deduce that 
\[
\underset{k\in\ints_{\geq1}}{\sup}G_{p}(k)
\]
can be approximated by $\sum_{l\in L}G_{p}(l)$ for some finite subset
$L\subseteq\ints_{\geq1}$, with $L$ independent of $p$. 
\item[(d)] To deal with $G_{p}(l)$ for $l\in L$, we use a transfer principle
for boundedness of motivic functions from \cite{CGH16} (see Theorem
\ref{thm:-transfer principle for bounds} below) to reduce to a question
about the $\mathbb{F}_{p}$-fibers of the $(l-1)$-th jet of $\varphi$.
We then combine Lang-Weil type arguments on the jets of $\varphi$,
together with a jet-scheme interpretation of the (FRS) property (Proposition
\ref{prop:jet scheme description of the (FRS) property}), to deduce
that $G_{p}(l)<C$ for $p\gg1$, $l\in L$ and some constant $C>0$
independent of $p$. 
\end{enumerate}
This shows $(1)\Rightarrow(2)$. To prove $(1)\Rightarrow(3)$, we
approximate $\widetilde{h}_{p}$ with a motivic function $h_{p}$,
which unlike $\widetilde{h}_{p}$, is formally non-negative. We then
apply similar steps as above (with a few extra complications) to $h_{p}$.

\subsection{Further discussion}

The (FRS) property was first introduced and studied in \cite{AA16,AA18},
where a very useful analytic interpretation was given as follows.
Given a morphism $\varphi:X\rightarrow Y$ between smooth $\rats$-varieties,
the (FRS) property of $\varphi$ is characterized by the property
that for every $F\in\mathrm{Loc}_{0}$ and every smooth, compactly
supported measure $\mu_{X(F)}$ on $X(F)$, the pushforward measure
$\varphi_{*}(\mu_{X})$ on $Y(F)$ has continuous density (see Theorem
\ref{thm:analytic criterion of the (FRS) property} or \cite[Theorem 3.4]{AA16}).
Our number theoretic characterization (Theorem \ref{Thm A}) can be
seen as a refinement of this analytic characterization.

\medskip{}

These characterizations allow one to use algebro-geometric tools to
solve various problems in analysis, probability and group theory.
For a motivating example, let $\underline{G}$ be a semisimple algebraic
$\rats$-group and let $\varphi_{\mathrm{comm}}^{*t}:\underline{G}^{2t}\rightarrow\underline{G}$
be the map $(g_{1},...,g_{2t})\mapsto[g_{1},g_{2}]\cdot...\cdot[g_{2t-1},g_{2t}]$,
corresponding to the product of $t$ commutator maps. Using the above
characterizations and a theorem of Frobenius, one has: 
\[
\varphi_{\mathrm{comm}}^{*t}\text{ is (FRS)}\Rightarrow\#\{N\text{-dimensional irreducible \ensuremath{\mathbb{C}}-representations of }\underline{G}(\Zp)\}=O(N^{2t-2}).\tag{\ensuremath{\star}}
\]
Aizenbud and Avni showed in \cite{AA16,AA18}, that $\varphi_{\mathrm{comm}}^{*21}$
is (FRS) for every $\underline{G}$ as above, which via ($\star$),
confirmed a conjecture of Larsen-Lubotzky \cite{LL08} about representation
growth of compact $p$-adic and arithmetic groups. These bounds were
improved in \cite{Bud19,Kap,GHb}.

\medskip{}

The above situation can be generalized as follows. Let $\varphi:X\rightarrow\underline{G}$
be a dominant morphism from a smooth $\rats$-variety $X$ to a connected
algebraic group $(\underline{G},\cdot_{\underline{G}})$. We define
the \emph{self-convolution} $\varphi*\varphi:X\times X\rightarrow\underline{G}$
of $\varphi$ by $\varphi*\varphi(x_{1},x_{2})=\varphi(x_{1})\cdot_{\underline{G}}\varphi(x_{2})$,
and write $\varphi^{*t}:X^{t}\rightarrow G$ for the $t$-th convolution
power of $\varphi$. Similarly to the usual convolution operation
in analysis, this algebraic convolution operation has a smoothing
effect on morphisms; In \cite{GH19,GHa}, it was shown that $\varphi^{*t}:X^{t}\rightarrow\underline{G}$
has increasingly better singularity properties as $t$ grows, and
eventually, $\varphi^{*t}$ becomes (FRS) for every $t\geq t_{\mathrm{0}}$,
for some $t_{0}\in\nats$. \medskip{}

Moving to the probabilistic picture, let $\mu_{X(\Zp)}$ and $\mu_{\underline{G}(\Zp)}$
be the canonical measures on $X(\Zp)$ and $\underline{G}(\Zp)$,
normalized to have total mass $1$. One can then study the collection
of random walks on $\underline{G}(\Zp)$, induced by the pushforward
measures $\{\varphi_{*}\mu_{X(\Zp)}\}_{p\in\mathrm{primes}}$, by
analyzing the convergence rate of their self-convolutions $(\varphi_{*}\mu_{X(\Zp)})^{*t}$
to $\mu_{\underline{G}(\Zp)}$, in the $L^{q}$-norm ($q\geq1$).
This rate of convergence can be measured by the notion of \emph{$L^{q}$-mixing
time} (see e.g. \cite[Chapter 4]{LeP17}). Note that the analytic
convolution operation commutes with the algebraic convolution defined
above, so that $(\varphi_{*}\mu_{X(\Zp)})^{*t}=(\varphi^{*t})_{*}\mu_{X^{t}(\Zp)}$.
This makes Theorem \ref{Thm A} the connecting link between the algebraic
and the probabilistic pictures above. \medskip{}

Explicitly, let us denote by $t_{\mathrm{alg}}$ the minimal $t\in\nats$
such that $\varphi^{*t}$ is (FRS) and has geometrically irreducible
fibers, and call it the \emph{algebraic mixing time of} $\varphi$.
Then Theorem \ref{Thm A}, and its general form Theorem \ref{thm:number theoretic characterization of the (FRS) property},
imply that the algebraic mixing time of $\varphi$ is equal to the
uniform (in $p\gg1$) $L^{\infty}$-mixing time of the random walks
on $\left\{ \underline{G}(\Zp)_{p}\right\} $ induced by $\left\{ \varphi_{*}\mu_{X(\Zp)}\right\} _{p}$
(see \cite[Definition 9.2]{GHb}). This philosophy was implemented
in \cite{GHb}, which motivated this work. There, the authors analyzed
the singularity properties of word maps on semi-simple algebraic groups,
using purely algebraic techniques, and obtained probabilistic results
on word measures. In particular, Theorem \ref{Thm A} completes the
proof of \cite[Theorems G and 9.3(2)]{GHb}.

\subsection{Conventions}
\begin{itemize}
\item Throughout the paper, we use $K,K',K''$ to denote number fields and
$\mathcal{O}_{K},\mathcal{O}_{K'},\mathcal{O}_{K''}$ for their rings
of integers. Similarly, local fields and their rings of integers are
denoted by $F,F',F''$ and $\mathcal{O}_{F},\mathcal{O}_{F'},\mathcal{O}_{F''}$,
respectively. 
\item Given a local ring $A$, a morphism $\varphi:X\rightarrow Y$ of schemes
$X$ and $Y$, and given $y\in Y(A)$ (i.e. a morphism $\spec(A)\rightarrow Y$),
we denote by $X_{y,\varphi}:=\spec(A)\times_{Y}X$ the scheme theoretic
fiber over $y$, and simply by $\varphi^{-1}(y)\subseteq X(A)$ the
set theoretic fiber of the induced map $\varphi:X(A)\rightarrow Y(A)$.
Note that if $y\in Y$ is a schematic point, then it can be viewed
as $y\in Y(\kappa(y))$, where $\kappa(y)$ is the residue field of
$y$, so that $X_{y,\varphi}:=\spec(\kappa(y))\times_{Y}X$. 
\item Given a $K$-morphism $\varphi:X\rightarrow Y$ between $K$-varieties
$X$ and $Y$, we denote by $X^{\mathrm{sm}}$ (resp.~$X^{\mathrm{sing}}$)
the smooth (resp.~non-smooth) locus of $X$. We denote by $X^{\mathrm{sm,\varphi}}$
(resp.~$X^{\mathrm{sing,\varphi}}$) the smooth (resp.~non-smooth)
locus of $\varphi$ in $X$. 
\item We denote the base change of an $S$-scheme $X$ with respect to $S'\rightarrow S$
by $X_{S'}$. 
\end{itemize}
\begin{acknowledgement*}
The author R.\,C.~was partially supported by the European Research
Council under the European Community's Seventh Framework Programme
(FP7/2007-2013) with ERC Grant Agreement nr. 615722 MOTMELSUM, KU
Leuven IF C14/17/083, and Labex CEMPI (ANR-11-LABX-0007-01). The author
I.\,G.~was partially supported by ISF grant 249/17, BSF grant 2018201
and by a Minerva foundation grant. \\ The authors wish 
to thank Rami Aizenbud, Nir Avni, Jan Denef and Julien Sebag for many useful discussions.
\end{acknowledgement*}

\section{Preliminaries}

\subsection{Jet schemes and singularities}

For a thorough discussion of jet schemes see \cite[Chapter 3]{CLNS18}
and \cite{EM09}. 
\begin{defn}[{{{cf.~\cite[Section 3.2]{CLNS18}}}}]
\label{def:basic definition jet schemes}Let $S$ be a scheme and
let $X$ be a scheme over $S$. 
\begin{enumerate}
\item For each $k\in\nats$, we define the \emph{$k$-th jet scheme} of
$X$, denoted $J_{k}(X/S)$ as the $S$-scheme representing the functor
\[
\mathcal{J}_{k}(X/S):W\longmapsto\mathrm{Hom}_{S\text{-schemes}}(W\times_{\spec\ints}\spec(\ints[t]/(t^{k+1})),X),
\]
where $W$ is an $S$-scheme. We write $J_{k}(X)$ if the scheme $S$
is understood. 
\item Given an $S$-morphism $\varphi:X\rightarrow Y$ and an $S$-scheme
$W$, the composition with $\varphi$ induces a map $\mathcal{J}_{k}(X/S)(W)\rightarrow\mathcal{J}_{k}(Y/S)(W)$,
which yields a morphism 
\[
J_{k}(\varphi):J_{k}(X/S)\rightarrow J_{k}(Y/S),
\]
called the \textsl{$k$-th jet} of $\varphi$. 
\item For any $k_{1}\geq k_{2}\in\nats$ the reduction map $\ints[t]/(t^{k_{1}+1})\rightarrow\ints[t]/(t^{k_{2}+1})$
induces a natural collection of morphisms $\pi_{k_{2},X}^{k_{1}}:J_{k_{1}}(X/S)\rightarrow J_{k_{2}}(X/S)$
which are called \textit{truncation maps}. Note that the collection
$\{J_{k}(\varphi):J_{k}(X/S)\rightarrow J_{k}(Y/S)\}_{k\in\nats}$
commutes with $\{\pi_{n,X}^{m}\}_{m\geq n}$. 
\item The natural map $\ints\rightarrow\ints[t]/(t^{m+1})$ induces a \textsl{zero
section} $s_{m,X}:X\hookrightarrow J_{m}(X)$. We sometimes write
$\pi_{n}^{m}$ and $s_{m}$ instead of $\pi_{n,X}^{m}$ and $s_{m,X}$,
when $X$ is clear. 
\end{enumerate}
\end{defn}

In the rest of this subsection, we assume $S=\mathrm{Spec}K$. Musta\c{t}\u{a}
gave the following interpretation of rational singularities in terms
of jet schemes: 
\begin{thm}[\cite{Mus01}]
\label{thm:rational singularities and jet schemes}Let $X$ be a
geometrically irreducible, local complete intersection $K$-variety,
with $\mathrm{char}(K)=0$. Then $J_{k}(X)$ is geometrically irreducible
for all $k\geq1$ if and only if $X$ has rational singularities. 
\end{thm}

Using Theorem \ref{thm:rational singularities and jet schemes}, one
can obtain a similar characterization of (FRS) morphisms: 
\begin{prop}[{\cite[Corollary 3.12]{GHb} and \cite{Ish09}}]
\label{prop:jet scheme description of the (FRS) property}Let $X$
and $Y$ be smooth, geometrically irreducible $K$-varieties, and
let $\varphi:X\rightarrow Y$ be a $K$-morphism. 
\begin{enumerate}
\item Assume $\mathrm{char}(K)=0$. Then the morphism $\varphi$ is (FRS)
if and only if $J_{k}(\varphi)$ is flat, with locally integral fibers
for each $k\in\nats$. 
\item The morphism $\varphi$ is smooth if and only if $J_{k}(\varphi)$
is smooth for each $k\in\nats$. 
\end{enumerate}
\end{prop}

\begin{rem}
\label{rem:jets of an affine scheme}Let $k$ be a natural number,
and $K$ be a field with $\mathrm{char}(K)=0$ or $\mathrm{char}(K)\gg1$
(in terms of $k$). Then the jet scheme $J_{k}(X)$ of an affine $K$-scheme
$X\subseteq\mathbb{A}^{n}$ has a simple description; write $X=\spec K[x_{1},\dots,x_{n}]/(f_{1},\dots,f_{l})$.
Then
\[
J_{k}(X)=\spec K[x_{1},\dots,x_{n},x_{1}^{(1)},\dots,x_{n}^{(1)},\dots,x_{1}^{(k)},\dots,x_{n}^{(k)}]/(\{f_{j}^{(u)}\}_{j=1,u=0}^{l,k}),
\]
where $f_{i}^{(u)}$ is the $u$-th formal derivative of $f_{i}$.
For example, if $f=x_{1}x_{2}^{2}$ then $f^{(1)}=x_{1}^{(1)}x_{2}^{2}+2x_{1}x_{2}x_{2}^{(1)}$.
Similarly, $J_{k}(\varphi)=(\varphi,\varphi^{(1)},...,\varphi^{(k)})$
for a morphism $\varphi:X\rightarrow Y$ of affine $K$-schemes. 
\end{rem}

The next proposition will be useful in Section \ref{sec:Number-theoretic-characterizatio}. 
\begin{prop}
\label{prop:useful lemma for jet-flat maps}Let $k\in\nats$ and let
$\varphi:X\rightarrow Y$ be $K$-morphism as in Proposition \ref{prop:jet scheme description of the (FRS) property},
with $\mathrm{char}(K)=0$ or $\mathrm{char}(K)\gg1$ (in terms of
$k$). Then $J_{k}(X)^{\mathrm{sm},J_{k}(\varphi)}=J_{k}(X^{\mathrm{sm},\varphi})$. 
\end{prop}

\begin{proof}
It follows from Proposition \ref{prop:jet scheme description of the (FRS) property}(2)
that $J_{k}(X^{\mathrm{sm},\varphi})\subseteq J_{k}(X)^{\mathrm{sm},J_{k}(\varphi)}$,
so it is left to show the other inclusion. We may assume that $X$
and $Y$ are affine, and that $Y$ admits an \'etale map $\psi:Y\rightarrow\mathbb{A}_{K}^{m}$.
We may further assume that $Y=\mathbb{A}_{K}^{m}$. Indeed, we have:
\[
J_{k}(X)^{\mathrm{sm},J_{k}(\varphi)}=J_{k}(X)^{\mathrm{sm},J_{k}(\psi\circ\varphi)}\text{ and }J_{k}(X^{\mathrm{sm},\psi\circ\varphi})=J_{k}(X^{\mathrm{sm},\varphi}).
\]
By Remark \ref{rem:jets of an affine scheme}, we can write $X=\mathrm{Spec}K[x_{1},...,x_{n+l}]/(f_{1},...,f_{l})$,
and 
\[
J_{k}(X)=\spec K[x_{1},\dots,x_{n+l},\dots,x_{1}^{(k)},\dots,x_{n+l}^{(k)}]/(\{f_{j}^{(u)}\}_{j=1,u=0}^{l,k}).
\]
Moreover $J_{k}(\varphi)=(\varphi,\varphi^{(1)},...,\varphi^{(k)})$
where $\varphi=(f_{l+1},...,f_{l+m}):X\rightarrow\mathbb{A}_{K}^{m}$.
Write $F_{u(l+m)+j}:=f_{j}^{(u)}$ and $X_{u(n+l)+i}:=x_{i}^{(u)}$,
and let $\overline{a}:=(a,a^{(1)},...,a^{(n+l)})\in J_{k}(X)$. Then
$J_{k}(\varphi)$ is smooth at $\overline{a}$ if and only if the
matrix $M=\left(\frac{\partial F_{j}}{\partial X_{i}}|_{\overline{a}}\right)_{i=1,j=1,}^{(n+l)(k+1),(l+m)(k+1)}$
is of full rank $(l+m)(k+1)$. Note that $M$ has the shape 
\[
M=\left(\begin{array}{cccc}
M_{00} & M_{01} & ... & M_{0k}\\
0 & M_{11} & ... & ...\\
... & ... & ... & ...\\
0 & ... & 0 & M_{kk}
\end{array}\right),
\]
where $M_{u_{1}u_{2}}=\left(\frac{\partial f_{j}^{(u_{2})}}{\partial x_{i}^{(u_{1})}}|_{\overline{a}}\right)_{i=1,j=1}^{(n+l),(l+m)}$
for $0\leq u_{1}\leq u_{2}\leq k$. If $M$ is of full rank, then
also $M_{00}=\left(\frac{\partial f_{j}}{\partial x_{i}}|_{a}\right)_{i=1,j=1}^{(n+l),(l+m)}$
must be of full rank, which in turn implies that $\varphi$ is smooth
at $a$, and the proposition follows. 
\end{proof}
\begin{rem}
The case $Y=\mathbb{A}^{1}$ of Proposition \ref{prop:useful lemma for jet-flat maps}
has essentially been proven in \cite[proof of Theorem 3.3]{EMY03}
and \cite[Proposition 4.12]{Mus01} (see also \cite[p.222]{Ish18}).
Proposition \ref{prop:useful lemma for jet-flat maps} also relates
to \cite[Questions 4.10 and 4.11]{Mus01}, as follows. Given a local
complete intersection variety $X$, it can be written, locally, as
a fiber $\widetilde{X}_{0,\varphi}$ of a flat morphism $\varphi:\widetilde{X}\rightarrow\mathbb{A}^{m}$,
with $\widetilde{X}$ smooth. If we assume that $J_{k}(\varphi)$
is flat for all $k$, then Proposition \ref{prop:useful lemma for jet-flat maps}
combined with \cite[III, Theorem 10.2]{Har77} implies that $(\pi_{0,\widetilde{X}_{0,\varphi}}^{k})^{-1}((\widetilde{X}_{0,\varphi})^{\mathrm{sm}})=J_{k}(\widetilde{X}_{0,\varphi})^{\mathrm{sm}}$
for all $k$, which gives a positive answer to \cite[Question 4.11]{Mus01}
in this case. If $J_{k}(\varphi)$ is not flat, one can still effectively
describe its smooth locus, but it is harder to describe the smooth
locus of its fibers. 
\end{rem}

\subsubsection{\label{subsec:E-smooth-and--jet flat}$E$-smooth and $\varepsilon$-jet
flat morphisms}

We next introduce several properties of morphisms between smooth varieties:
$\varepsilon$-flatness, $\varepsilon$-jet flatness, and $E$-smoothness.
The first two notions were first introduced in \cite{GHb}, whereas
the $E$-smoothness notion is new.
\begin{defn}
\label{def:epsilon jet flat}Let $X$ and $Y$ be smooth, geometrically
irreducible $K$-varieties, and let $\varphi:X\rightarrow Y$ be a
$K$-morphism, let $E\geq1$ be an integer and let $\varepsilon\in\reals_{>0}$.
Then:
\begin{enumerate}
\item $\varphi$ is called \emph{$\varepsilon$-flat} if for every $x\in X$
we have $\mathrm{dim}X_{\varphi(x),\varphi}\leq\mathrm{dim}X-\varepsilon\mathrm{dim}Y$. 
\item $\varphi$ is called \emph{$\varepsilon$-jet flat} (resp.~jet-flat)
if $J_{k}(\varphi)$ is $\varepsilon$-flat (resp.~flat) for every
$k\in\nats$. 
\item A jet-flat morphism $\varphi$ is called \emph{$E$-smooth} if for
all $k\in\ints_{\geq0}$ and all $\widetilde{x}\in J_{k}(X)$, the
set $(J_{k}(X)_{J_{k}(\varphi)(\widetilde{x}),J_{k}(\varphi)})^{\mathrm{sing}}$
is of codimension at least $E$ in $J_{k}(X)_{J_{k}(\varphi)(\widetilde{x}),J_{k}(\varphi)}$.
\end{enumerate}
\end{defn}

\begin{rem}
\label{rem:relating jet flatness}~
\begin{enumerate}
\item By \cite{Mus02}, a morphism $\varphi$ as in Definition \ref{def:epsilon jet flat} is $\varepsilon$-jet
flat if and only if $\mathrm{lct}(X,X_{\varphi(x),\varphi})\geq\varepsilon\mathrm{dim}Y$
for all $x\in X$, where $\mathrm{lct}(X,X_{\varphi(x),\varphi})$
is the log-canonical threshold of the pair $(X,X_{\varphi(x),\varphi})$. 
\item In addition, it follows from \cite{Mus01,EM04} (see \cite[Corollary 3.12]{GHb})
that if $\varphi$ is a normal morphism, then it is jet-flat if and
only if it is flat and has fibers with log-canonical singularities.
\end{enumerate}
\end{rem}

$\varepsilon$-flatness is a quantitative way to measure how close
a morphism between smooth varieties is to being flat. Similarly, $\varepsilon$-jet
flatness measures how close a morphism is to being jet-flat, which
is very close to being an (FRS)-morphism. On the other hand, the starting
point of $E$-smoothness is when $\varphi$ is jet-flat, and the larger
$E$ is, the better the singularities of $\varphi$ are. This is illustrated
in the next lemma:
\begin{lem}
\label{lem:1-smooth is (FRS)}Let $\varphi:X\rightarrow Y$ be $K$-morphism
between smooth, geometrically irreducible $K$-varieties.
\begin{enumerate}
\item $\varphi$ is $1$-smooth if and only if $\varphi$ is (FRS). 
\item $\varphi$ is $2$-smooth if and only if $\varphi$ is flat with fibers
of terminal singularities.
\end{enumerate}
\end{lem}

\begin{proof}
By Proposition \ref{prop:jet scheme description of the (FRS) property},
$\varphi$ is (FRS) if and only if $J_{k}(\varphi)$ is flat, with
locally integral fibers for each $k\in\nats$. By \cite[Corollary 3.12(3)]{GHb},
$\varphi$ is flat with fibers of terminal singularities if and only
if $J_{k}(\varphi)$ is flat, with normal fibers for each $k\in\nats$.
In particular, in the situation of (1) and (2), $\varphi$ is always
jet-flat, and thus the fibers of $J_{k}(\varphi)$ are local complete
intersection, and hence Cohen-Macaulay. Serre's criterion for normality
and reducedness \cite[Proposition 5.8.5, Theorem 5.8.6]{Gro66} and
\cite[Proposition 1.4]{Mus01} now imply Items (1) and (2).
\end{proof}

\subsection{\label{subsec:Motivic-functions}Motivic functions}

In this subsection we recall the definition and some properties of
motivic functions. In order to prove Theorem \ref{Thm A}, we encode
the collection $\left\{ \#\varphi^{-1}(y)\right\} _{p,k,y\in Y(\ints/p^{k}\ints)}$
using a single motivic function, and utilize this to obtain the desired
uniform bounds. We use the notion of motivic functions as was defined
and studied in \cite{CL08,CL10,CGH14,CGH16}. In order to fully exploit
the advantages of the motivic realm, we introduce the class of \emph{formally
non-negative motivic functions}, which is the specialization to local
fields of \cite[Section 5.3]{CL08}.

Throughout this subsection, we fix a number field $K$. We use the
(three-sorted) \textsl{Denef-Pas language, }denoted 
\[
\Ldp=(\mathcal{L}_{\mathrm{Val}},\mathcal{L}_{\mathrm{Res}},\mathcal{L}_{\mathrm{Pres}},\mathrm{\val},\mathrm{\ac}),
\]
where: 
\begin{enumerate}
\item The valued field sort $\VF$ is endowed with the language of rings
$\mathcal{L}_{\mathrm{Val}}$, with coefficients in $\mathcal{O}_{K}$. 
\item The residue field sort $\RF$ is endowed with the language of rings
$\mathcal{L}_{\mathrm{Res}}$. 
\item The value group sort $\VG$ (which we just call $\ints$), is endowed
with the Presburger language $\mathcal{L}_{\mathrm{Pres}}=(+,-,\leq,\{\equiv_{\mathrm{mod}~n}\}_{n>0},0,1)$
of ordered abelian groups along with constants $0,1$ and a family
of relations $\{\equiv_{\mathrm{mod}~n}\}_{n>0}$ of congruences modulo
$n$. 
\item $\val:\VF\backslash\{0\}\rightarrow\ints$ and $\ac:\VF\rightarrow\RF$
are two function symbols. 
\end{enumerate}
Let $\mathrm{Loc}$ be the collection of all non-Archimedean local
fields $F$ with a ring homomorphism $\mathcal{O}_{K}\rightarrow F$.
We denote by $\mathrm{Loc}_{0}$ (resp. $\mathrm{Loc}_{+}$) the collection
of all $F\in\mathrm{Loc}$ of characteristic zero (resp. positive
characteristic). For $F\in\mathrm{Loc}$, we denote by $\mathcal{O}_{F}$
its ring of integer, by $k_{F}$ its residue field, and by $q_{F}$
the number of elements in $k_{F}$. We use the notation $\mathrm{Loc}_{\gg}$\footnote{Our notation for $\mathrm{Loc}_{\gg}$ is slightly more restrictive
than the one used in \cite{CGH18}. Here $\mathrm{Loc}_{\gg}$ consists
of $\mathrm{Loc}_{0,\gg}\cup\mathrm{Loc}_{+,\gg}$ while in \cite{CGH18},
it consisted of $\mathrm{Loc}_{0}\cup\mathrm{Loc}_{+,\gg}$.} (resp. $\mathrm{Loc}_{0,\gg}$, $\mathrm{Loc}_{+,\gg}$), for the
collection of $F\in\mathrm{Loc}$ (resp. $\mathrm{Loc}_{0}$, $\mathrm{Loc}_{+}$)
with large enough residual characteristic (depending on some given
data).

Given $F\in\mathrm{Loc}$ (and a chosen uniformizer $\varpi_{F}$
of $\mathcal{O}_{F}$), we can interpret $\val$ and $\ac$ as the
valuation map $\val:F^{\times}\rightarrow\ints$ and the angular component
map $\ac:F\rightarrow k_{F}$, where $\ac(0)=0$ and $\ac(x)=x\cdot\varpi_{F}^{-\val(x)}\mod\,\varpi_{F}\mathcal{O}_{F}$
for $x\neq0$. Hence, any formula $\phi$ in $\Ldp$ with $n_{1}$
free $\VF$-variables, $n_{2}$ free $\RF$-variables and $n_{3}$
free $\ints$-variables, yields a subset $\phi(F)\subseteq F^{n_{1}}\times k_{F}^{n_{2}}\times\mathbb{Z}^{n_{3}}$.
A collection $X=(X_{F})_{F\in\mathrm{Loc}_{\gg}}$ with $X_{F}=\phi(F)$
is called an \textsl{$\mathcal{L}_{\mathrm{DP}}$-definable set}.
Given \textsl{$\mathcal{L}_{\mathrm{DP}}$-}definable sets $X$ and
$Y$, an \textsl{$\mathcal{L}_{\mathrm{DP}}$-}\textit{definable function}
is a collection $f=(f_{F}:X_{F}\rightarrow Y_{F})_{F\in\mathrm{Loc}_{\gg}}$
of functions whose collection of graphs is a definable set. We will
often say ``definable'' instead of \textsl{``$\mathcal{L}_{\mathrm{DP}}$}-definable''. 
\begin{defn}[{See \cite[Subsections 4.2.4-4.2.5]{CGH14}}]
\label{def:Presburger constructible functions}Let $X$ be an $\Ldp$-definable
set. A collection $f=(f_{F})_{F\in\mathrm{Loc}_{\gg}}$ of functions
$f_{F}:X_{F}\rightarrow\mathbb{R}$ is called a \textit{Presburger}\textit{\emph{
}}\emph{constructible}\textit{ function}, if it can be written as
\[
f_{F}(x)=\sum\limits _{i=1}^{N_{1}}q_{F}^{\alpha_{i,F}(x)}\prod\limits _{j=1}^{N_{2}}\beta_{ij,F}(x)\prod\limits _{j=1}^{N_{3}}\frac{1}{1-q_{F}^{a_{ij}}},
\]
where $N_{1},N_{2},N_{3}$ and $a_{il}$ are non-zero integers, and
$\alpha_{i},\beta_{ij}:X\to\ints$ are definable functions. Given
$f$ as above, set $\widetilde{f}_{F}:X_{F}\times\reals_{>1}\rightarrow\reals$
by 
\[
\widetilde{f}_{F}(x,s):=\sum\limits _{i=1}^{N_{1}}s^{\alpha_{i,F}(x)}\prod\limits _{j=1}^{N_{2}}\beta_{ij,F}(x)\prod\limits _{j=1}^{N_{3}}\frac{1}{1-s^{a_{ij}}}.
\]
We say that $f$ is \emph{formally non-negative} if $\widetilde{f}_{F}$
takes non-negative values for every $F\in\mathrm{Loc}_{\gg}$. We
denote by $\mathcal{P}(X)$ the ring of Presburger constructible functions
on $X$, and by $\mathcal{P}_{+}(X)$ the sub-semiring of formally
non-negative functions. 
\end{defn}

\begin{defn}
\label{def:motivic functions}Let $X$ be an $\Ldp$-definable set.
A collection $h=(h_{F})_{F\in\mathrm{Loc}_{\gg}}$ of functions $h_{F}:X_{F}\rightarrow\mathbb{R}$
is called a \textit{motivic}\textit{\emph{ }}\emph{function}, if it
can be written as: 
\[
h_{F}(x)=\sum\limits _{i=1}^{N}\#Y_{i,F,x}\cdot f_{i}{}_{F}(x),
\]
where: 
\begin{itemize}
\item $Y_{i,F,x}=\{\xi\in k_{F}^{r_{i}}:(x,\xi)\in Y_{i,F}\}$ is the fiber
over $x\in X_{F}$ of a definable set $Y_{i}\subseteq X\times\mathrm{RF}^{r_{i}}$
with $r_{i}\in\nats$. 
\item Each $f_{i}$ is a Presburger constructible function. 
\end{itemize}
If furthermore every $f_{i}$ is formally non-negative, then we call
$h$ a \textit{formally }\emph{non-negative}\textit{ motivic}\textit{\emph{
}}\emph{function. }We denote by $\mathcal{C}(X)$ the ring of motivic
functions on $X$, and by $\mathcal{C}_{+}(X)$ the sub-semiring of
formally non-negative motivic functions. 
\end{defn}

The classes $\mathcal{C}(X)$ and $\mathcal{C}_{+}(X)$ defined above
are the specialization to local fields of more abstract classes of
motivic functions defined in \cite[Section 5]{CL08} (e.g. see the
discussion in \cite[Section 4.2]{CGH14}). In \cite[Theorem 10.1.1]{CL08},
it is shown that these more general classes are preserved under a
formal integration operation, and in \cite[Section 9]{CL10} it is
shown that this formal integration operation commutes with usual $p$-adic
integration under specialization. This implies the following theorem: 
\begin{thm}
\label{integration of motivic}Let $X$ be an $\Ldp$-definable set,
and let $f$ be in $\mathcal{C}_{+}(X\times\mathrm{VF}^{m})$. Assume
that for every $F\in\mathrm{Loc}_{\gg}$ and every $x\in X_{F}$,
the function $y\mapsto f_{F}(x,y)$ belongs to $L^{1}(F^{m})$. Then
there exists $g$ in $\mathcal{C}_{+}(X)$ such that 
\begin{equation}
g_{F}(x)=\int_{y\in F^{m}}f_{F}(x,y)\left|dy\right|.\label{eq:integration}
\end{equation}
\end{thm}

\begin{rem}
\label{rem:stronger formulation of Theorem 2.8}In \cite[Theorem 4.3.1]{CGH14}
it was shown that the class of motivic functions is preserved under
integration in the following stronger sense, namely, that given $f$
in $\mathcal{C}(X\times\mathrm{VF}^{m})$, one can find $g\in\mathcal{C}(X)$
such that for every $F\in\mathrm{Loc}_{\gg}$ and $x\in X_{F}$, if
$y\mapsto f_{F}(x,y)$ belongs to $L^{1}(F^{m})$ then (\ref{eq:integration})
holds. This stronger statement relies on an interpolation theorem
\cite[Theorem 4.3.3]{CGH14} for functions in $\mathcal{C}(X)$. It
would be interesting to prove a similar interpolation result for the
class of formally non-negative motivic functions. This will imply
the stronger formulation of Theorem \ref{integration of motivic}
as in \cite[Theorem 4.3.1]{CGH14}. 
\end{rem}

Finally, we need the following transfer result between $\mathrm{Loc}_{0,\gg}$
and $\mathrm{Loc}_{+,\gg}$. 
\begin{thm}[{{{Transfer principle for bounds, \cite[Theorem 3.1]{CGH16}}}}]
\label{thm:-transfer principle for bounds}Let $X$ be an $\mathcal{L}_{\mathrm{DP}}$-definable
set, and let $H,G\in\mathcal{C}(X)$ be motivic functions. Then the
following holds for $F\in\mathrm{Loc}_{\gg}$; if 
\[
\left|H_{F}(x)\right|\leq\left|G_{F}(x)\right|,
\]
for each $x\in X_{F}$, then also 
\[
\left|H_{F'}(x)\right|\leq\left|G_{F'}(x)\right|,
\]
for every $F'\in\mathrm{Loc}$ with the same residue field as $F$,
and each $x\in X_{F'}$. 
\end{thm}

\section{An improvement of the approximation of suprema}

The main goal of this section is to show the following improvement
of \cite[Theorem 2.1.3]{CGH18} on approximate suprema. This improvement
is made possible by placing ourselves in the special case of formally
non-negative motivic functions and is not possible in the more general
situation of \cite{CGH18}. 
\begin{thm}[Improved approximation of suprema]
\label{thm:improved supremum theorem}Let $f$ be in $\mathcal{C}_{+}(X\times W)$,
where $X$ and $W$ are definable sets. Then there exist a constant
$C>0$, and a function $G\in\mathcal{C}_{+}(X)$ such that for any
$F\in\mathrm{Loc}_{\gg}$ and any $x\in X_{F}$ such that $w\mapsto f_{F}(x,w)$
is bounded on $W_{F}$, we have 
\[
\underset{w\in W_{F}}{\sup}f_{F}(x,w)\leq G_{F}(x)\leq C\cdot\underset{w\in W_{F}}{\sup}f_{F}(x,w).
\]
\end{thm}

The following lemma is immediate: 
\begin{lem}
\label{lem:supremum sandwich}Let $\{f_{i}\}_{i=1}^{N}$ be in $\mathcal{C}_{+}(X\times W)$
and set $f=\stackrel[i=1]{N}{\sum}f_{i}$. Then for $F\in\mathrm{Loc}_{\gg}$,
one has: 
\[
\frac{1}{N}\sum_{i=1}^{N}\underset{w\in W_{F}}{\sup}f_{i}{}_{F}(x,w)\leq\underset{w\in W_{F}}{\sup}f_{F}(x,w)\leq\sum_{i=1}^{N}\underset{w\in W_{F}}{\sup}f_{i}{}_{F}(x,w).
\]
\end{lem}

Let $f$ be in $\mathcal{C}_{+}(X\times W)$. By Definition \ref{def:motivic functions},
we can write $f(x,w)=\stackrel[i=1]{N}{\sum}\#Y_{i,x,w}\cdot g_{i}(x,w)$,
where $g_{i}\in\mathcal{P}_{+}(X\times W)$ and $Y_{i}\subseteq X\times W\times\mathrm{RF}^{r_{i}}$.
Lemma \ref{lem:supremum sandwich} thus implies the following: 
\begin{cor}
\label{cor:definable partition}Let $f$ be in $\mathcal{C}_{+}(X\times W)$,
where $X$ and $W$ are definable sets. 
\begin{enumerate}
\item Let $X\times W=\bigsqcup\limits _{i=1}^{M}C_{i}$ be a definable partition
and set $f_{i}(x,w)=f(x,w)\cdot1_{C_{i}}$. Then it is enough to prove
Theorem \ref{thm:improved supremum theorem} for each $f_{i}$. 
\item It is enough to prove Theorem \ref{thm:improved supremum theorem}
for $f$ of the form $f=\#Y_{x,w}\cdot g(x,w)$ where $g\in\mathcal{P}_{+}(X\times W)$. 
\end{enumerate}
\end{cor}

\begin{rem}
The key case of Theorem \ref{thm:improved supremum theorem} is when
neither $X$ nor $W$ involve valued field variables. The reduction
to this case needs to be done with care. Naively, one can use quantifier
elimination to eliminate the valued field variables, but this is problematic
since it mixes the valued field variables of $X$ and $W$, making
it hard to take supremum over the variables of $W$. In order to elude
this problem, we will apply cell decomposition iteratively, first
taking care of the $W$ variables and then taking care of the $X$
variables. 
\end{rem}

\begin{proof}[Proof of Theorem \ref{thm:improved supremum theorem}]
Let $f(x,w)=\#Y_{x,w}\cdot g(x,w)$ for some $g\in\mathcal{P}_{+}(X\times W)$
and $Y\subseteq X\times W\times\mathrm{RF}^{r}$. Without loss of
generality, we may assume that $X=\VF^{n_{1}}\times\RF^{n_{2}}\times\VG^{n_{3}}$
and $W=\VF^{m_{1}}\times\RF^{m_{2}}\times\VG^{m_{3}}$ for some $n_{i}\ge0$
and $m_{i}\ge0$. We will first reduce to the case where there are
no valued field variables, using the following claim.

\begin{claim}\label{claim1}

We may assume that $X=\RF^{n_{2}}\times\VG^{n_{3}}$ and $W=\RF^{m_{2}}\times\VG^{m_{3}}$.

\end{claim}
\begin{proof}[Proof of Claim \ref{claim1}]
We first get rid of the valued field variables $\VF^{m_{1}}$ of
$W$. Without loss of generality we may assume that $W=\VF^{m_{1}}$.
By induction, we may further assume that $m_{1}=1$. By \cite[Theorem 7.2.1]{CL08}
there exists a definable surjection $\lambda:X\times W\rightarrow C\subseteq X\times\RF^{s}\times\ints^{r}$
over $X$ as well as $\psi\in\mathcal{C}_{+}(C)$ such that $f=\psi\circ\lambda$.
Note that 
\[
\underset{w\in W_{F}}{\sup}f_{F}(x,w)=\underset{w\in W_{F}}{\sup}\psi_{F}\circ\lambda_{F}(x,w)=\underset{(\xi,k)\in k_{F}^{s}\times\ints^{r}}{\sup}\psi_{F}(x,\xi,k),
\]
up to extending $\psi$ by zero outside $C$. We may therefore assume
that $W=\RF^{m_{2}}\times\VG^{m_{3}}$. We next get rid of the valued
field variables $\VF^{n_{1}}$ of $X$, denoted $y:=y_{1},...,y_{n_{1}}$.
Write $x=(y,\eta,t)\in X$ and $w=(\xi,s)\in W$, with $\RF$-variables
$\eta,\xi$ and $\VG$-variables $t,s$. By Definition \ref{def:motivic functions},
$f$ is determined by a finite collection $\alpha_{i},\beta_{ij}:X\times W\to\ints$
of definable functions, and by a definable set $Y\subseteq X\times W\times\mathrm{RF}^{r}$.
By quantifier elimination in the valued field variables \cite[Theorem 4.1]{Pas89},
there exist finitely many polynomials $g_{1},...,g_{l}\in\ints[y_{1},...,y_{n_{1}}]$
such that the graphs of the functions in $\{\alpha_{i},\beta_{ij}\}$
can be defined by formulas of the form 
\[
\bigvee_{i=1}^{L}\chi_{i}(\xi,\eta,\ac(g_{1}(y)),...,\ac(g_{l}(y)))\wedge\theta_{i}(t,s,t',\val(g_{1}(y)),...,\val(g_{l}(y))),
\]
and the subset $Y$ can be defined by a formula of the form 
\[
\bigvee_{i=1}^{L'}\widetilde{\chi}_{i}(\xi,\eta,\xi',\ac(g_{1}(y)),...,\ac(g_{l}(y)))\wedge\widetilde{\theta}_{i}(t,s,\val(g_{1}(y)),...,\val(g_{l}(y))),
\]
where $\chi_{i}$ and $\widetilde{\chi}_{i}$ are $\mathcal{L}_{\mathrm{Res}}$-formulas,
$\theta_{i}$ and $\widetilde{\theta}_{i}$ are $\mathcal{L}_{\mathrm{Pres}}$-formulas,
$t'$ is in $\ints$ and $\xi'$ is in $\RF^{r}$. We now set $\lambda':X\times W\rightarrow\RF^{s'}\times\ints^{r'}\times W$
by $\lambda'(x,w)=(\rho(x),w)$ with 
\[
\rho(x)=\rho(y,\eta,t):=(\eta,\ac(g_{1}(y)),...,\ac(g_{l}(y)),t,\val(g_{1}(y)),...,\val(g_{l}(y))).
\]
Let $C'$ be the image of $\lambda'$. Note we may find definable
functions $\widetilde{\alpha}_{i},\widetilde{\beta}_{ij}:C'\rightarrow\ints$
and a definable subset $\widetilde{Y}\subseteq C'\times\mathrm{RF}^{r}$
such that $\alpha_{i}=\widetilde{\alpha}_{i}\circ\lambda'$, $\beta_{i}=\widetilde{\beta}_{ij}\circ\lambda'$
and $Y=(\lambda'\times\mathrm{Id})^{-1}(\widetilde{Y})$. Using this
new definable data, we construct $\psi'\in\mathcal{C}_{+}(C')$ such
that $f=\psi'\circ\lambda'$ and again we have 
\[
\underset{w\in W_{F}}{\sup}f_{F}(x,w)=\underset{w\in W_{F}}{\sup}\psi'_{F}\circ\lambda'_{F}(x,w)=\underset{w\in W_{F}}{\sup}\psi'_{F}(\rho(x),w).
\]
Hence we have reduced to the case where $X=\RF^{n_{2}}\times\VG^{n_{3}}$.
This finishes the proof of Claim \ref{claim1}. 
\end{proof}
\begin{claim}\label{claim2}

We may assume that $X=\RF^{n_{2}}\times\VG^{n_{3}}$ and $W=\RF^{m_{2}}$. 

\end{claim}
\begin{proof}[Proof of Claim \ref{claim2}]
Write $x=(\eta,t)$ and $w=(\xi,s)$ for the variables of $X=\RF^{n_{2}}\times\VG^{n_{3}}$
and $W=\RF^{m_{2}}\times\VG^{m_{3}}$. We would like to get rid of
the value group variables $\VG^{m_{3}}$ of $W$. Using the (model
theoretic) orthogonality of the sorts $\VG$ and $\RF$, there is
a definable partition of $X\times W$, such that each definable part
$A$ is a box $A_{1}\times A_{2}$ with $A_{1}\subseteq\RF^{n_{2}}\times\RF^{m_{2}}$
and $A_{2}\subseteq\VG^{n_{3}}\times\VG^{m_{3}}$, and such that on
each $A$, $f$ has the form 
\[
f_{F}|_{A_{F}}(\eta,t,\xi,s)=\#Y_{\eta,\xi}\cdot H_{F}(t,s),
\]
for some $H\in\mathcal{P}_{+}(\VG^{n_{3}}\times\VG^{m_{3}})$ and
$Y\subseteq\RF^{n_{2}}\times\RF^{m_{2}}\times\mathrm{RF}^{r}$. By
Corollary \ref{cor:definable partition}, and by our assumption on
$A$, we may assume $f_{F}=\#Y_{\xi,\eta}\cdot H_{F}(t,s)$. Note
that for each $F\in\mathrm{Loc}_{\gg}$ and each $(\eta,t,\xi)\in X_{F}\times k_{F}^{m_{2}}$
one has 
\[
\underset{s\in\ints^{m_{3}}}{\sup}f_{F}(\eta,t,\xi,s)=\#Y_{\eta,\xi}\cdot\underset{s\in\ints^{m_{3}}}{\sup}H_{F}(t,s),
\]
In order to approximate $\underset{s\in\ints^{m_{3}}}{\sup}H_{F}(t,s)$,
it is enough to consider the case where $m_{3}=1$ and proceed by
induction on $m_{3}$. Using Presburger cell decomposition and rectilinearization
(see \cite[Theorems 1 and 3]{Clu03}) we may assume that $H$ is in
$\mathcal{P}_{+}(B)$ for $B\subseteq\VG^{n_{3}}\times\nats$ with
$B_{t}:=\{s\in\nats:(t,s)\in B\}$ is either a finite set for each
$t\in\ints^{n_{3}}$, or $B_{t}=\nats$, and moreover, $H$ is of
the form 
\[
H_{F}(t,s)=\sum_{i=1}^{N}c_{i,F}(t)s^{a_{i}}q_{F}^{b_{i}s},
\]
with $a_{i}\in\nats$ and $b_{i}\in\ints$ and $c_{i}$ in $\mathcal{P}(\VG^{n_{3}})$.
Denote by $T$ the image of projection of $B$ to $\VG^{n_{3}}$.
We repeat a part of the argument of the proof of \cite[Theorem 2.1.3]{CGH18}.
Namely, by \cite[Lemmas 2.2.3 and 2.2.4]{CGH18}, there exist $m,l\in\nats_{\geq1}$
and finitely many definable functions $h_{1},...,h_{l}:T\rightarrow\nats$
with $h_{j}(t)\in B_{t}$ such that for each $t\in T$ for which $s\mapsto H_{F}(t,s)$
is bounded on $B_{t}$, one has 
\[
\underset{s\in B_{t}}{\sup}H_{F}(t,s)\leq m\cdot\underset{1\leq j\leq l}{\max}H_{F}(t,h_{j}(t)).
\]
In particular, setting $\widetilde{H}(t):=m\cdot\sum_{j=1}^{l}H(t,h_{j}(t))\in\mathcal{P}_{+}(T)$
we get: 
\[
\underset{s\in B_{t}}{\sup}H_{F}(t,s)<\widetilde{H}_{F}(t)<m\cdot l\cdot\underset{s\in B_{t}}{\sup}H_{F}(t,s).
\]
This finishes the proof of Claim \ref{claim2}. 
\end{proof}
\begin{claim}\label{claim3}

We may assume that $X=\RF^{n_{2}}$ and $W=\RF^{m_{2}}$. 

\end{claim}
\begin{proof}
This follows directly by Claim \ref{claim2}, Corollary \ref{cor:definable partition},
and using the orthogonality of the sorts $\VG$ and $\RF$. 
\end{proof}
To continue the proof of Theorem \ref{thm:improved supremum theorem},
we may thus assume that $X=\RF^{n_{2}}$ and $W=\RF^{m_{2}}$. We
may assume, again using Corollary \ref{cor:definable partition},
that $f$ is of the form $f(x,w)=f(\eta,\xi)=u\cdot\#Y_{\eta,\xi}$,
with $\xi$ the coordinate on $W$, and $\eta$ on $X$ and $u=\{u_{F}\}_{F\in\mathrm{Loc}_{\gg}}$
is a motivic number. In particular, for each $\eta\in X_{F}$: 
\[
\underset{w\in W_{F}}{\sup}f_{F}(x,w)=\underset{\xi\in k_{F}^{m_{2}}}{\sup}f_{F}(\eta,\xi)=u_{F}\cdot\underset{\xi\in k_{F}^{m_{2}}}{\sup}\#Y_{\eta,\xi}.
\]
By a definable variant of the Lang-Weil estimates (see \cite[Main Theorem]{CvdDM92}),
there exists a definable partition $X\times W=\stackrel[i=0]{M}{\bigsqcup}A_{i}$
and constants $C'>0$, $d_{i}\in\nats$ and $l_{i1},l_{i2}\in\ints_{\geq1}$,
such that for each $1\leq i\leq M$ and each $F\in\mathrm{Loc}_{\gg}$
: 
\[
A_{i,F}:=\{(\eta,\xi)\in X_{F}\times W_{F}:\left|\#Y_{\eta,\xi}-\frac{l_{i1}}{l_{i2}}q_{F}^{d_{i}}\right|\leq C'\cdot q_{F}^{d_{i}-\frac{1}{2}}\},
\]
\[
A_{0,F}:=\{(\eta,\xi)\in X_{F}\times W_{F}:Y_{\eta,\xi}\text{ is empty}\}.
\]
Denote by $Z_{i}$ the projection of $A_{i}$ to $X$. For each subset
$I\subseteq\{1,\dots,M\}$, let $Z_{I}:=\bigcap_{i\in I}Z_{i}\backslash\bigcup_{j\in I^{c}}Z_{j}$,
with $Z_{\slashed{O}}:=X\backslash\stackrel[j=1]{M}{\bigcup}Z_{j}$.
Then $X=\bigsqcup_{I}Z_{I}$ is a definable partition, and thus we
may assume that $X=Z_{I}$. In this case, we have for $F\in\mathrm{Loc}_{\gg}$:
\begin{align*}
\underset{w\in W_{F}}{\sup}f_{F}(x,w) & =u_{F}\cdot\underset{\xi\in k_{F}^{m_{2}}}{\sup}\#Y_{\eta,\xi}\leq u_{F}\cdot\sum_{i\in I}\underset{\xi\in k_{F}^{m_{2}}}{\sup}\left(1_{A_{i,F}}\cdot\#Y_{\eta,\xi}\right)\\
 & \leq u_{F}\cdot\sum_{i\in I}2l_{i1}\cdot q_{F}^{d_{i}}\leq u_{F}\cdot\sum_{i\in I}4l_{i2}\cdot\underset{\xi\in k_{F}^{m_{2}}}{\sup}\left(1_{A_{i,F}}\cdot\#Y_{\eta,\xi}\right)\\
 & \leq\left(\sum_{i\in I}4l_{i2}\right)\underset{w\in W_{F}}{\sup}f_{F}(x,w),
\end{align*}
where we take zero for the supremum of the empty set. Since $\{u_{F}\cdot\sum_{i\in I}2l_{i1}\cdot q_{F}^{d_{i}}\}_{F\in\mathrm{Loc}_{\gg}}$
clearly lies in $\mathcal{C}_{+}(X)$, this finishes the proof of
Theorem \ref{thm:improved supremum theorem}. 
\end{proof}

\subsection{\label{subsec:Optimality-of-the bounds}Optimality of the bounds
and further remarks}

Let $X$ and $W$ be $\mathcal{L}_{\mathrm{DP}}$-definable sets.
Given a subclass $\mathcal{F}\subseteq\mathcal{C}(X\times W)$ of
motivic functions, one can ask whether for any $f\in\mathcal{F}$,
the function $\{\underset{w\in W_{F}}{\sup}f_{F}(x,w)\}_{F\in\mathrm{Loc}_{\gg}}$
can be approximated by a motivic function in $\mathcal{C}(X)$ up
to a constant $C$ in up to four increasing levels of approximation: 
\begin{enumerate}
\item[(1)] With $C$ depending on $F$ and $f$. 
\item[(2)] With $C$ depending on $f$ and independent of $F$. 
\item[(3)] With $C$ a universal constant, that is, uniform over all $f\in\mathcal{F}$
and $F\in\mathrm{Loc}_{\gg}$. 
\item[(4)] With $C=1+C'q_{F}^{-1/2}$ for some $C'$ depending on $f$ and independent
of $F$. 
\end{enumerate}
If the class $\mathcal{F}$ satisfies one of the Items $(i)$ above,
we say that $\mathcal{F}$ admits \emph{an approximation of suprema
of type $(i)$, or $\mathcal{F}$ is of type $(i)$. }Note that if
$\mathcal{F}$ is of type $(4)$ then it is also of type $(3)$, as
$C'q_{F}^{-1/2}<2$ for $F\in\mathrm{Loc}_{\gg}$. Similarly, type
$(i)$ is stronger than type $(j)$ for $j<i$. 
\begin{rem}
\label{rem:approximation}~ 
\begin{itemize}
\item The class $\mathcal{C}(X\times W)$ is not of type $(1)$ (and thus
of any type). Indeed, take $X=\ints^{2}$, $W=\{1,2\}\subseteq\ints$
and define $f\in\mathcal{C}(X\times W)$ by $f(x,y,1)=x^{2}-y$ and
$f(x,y,2)=y-x^{2}$. Then $\underset{w}{\sup\,}f(x,y,w)=\mathrm{max}(x^{2}-y,y-x^{2})$
cannot be approximated by a motivic function on $\mathcal{C}(X)$,
up to a constant depending on $F$ and $f$. 
\item The class $\mathcal{C}_{+}^{\mathrm{weak}}(X\times W):=\{f\in\mathcal{C}(X\times W):f_{F}\geq0\,\forall F\in\mathrm{Loc}_{\gg}\}$
is of type $(1)$, with $C=q_{F}^{C_{0}}$ for some $C_{0}>0$ depending
only on $f$. This is a special case treated in the proof of \cite[Theorem 2.1.3]{CGH18}.
One may wonder whether the class $\mathcal{C}_{+}^{\mathrm{weak}}(X\times W)$
is of type $(2)$. 
\end{itemize}
\end{rem}

Theorem \ref{thm:improved supremum theorem} shows that the family
$\mathcal{C}_{+}(X\times W)$, which is strictly contained in $\mathcal{C}_{+}^{\mathrm{weak}}(X\times W)$,
is of type $(2)$. This is the best possible approximation, as already
detected by the subclass $\mathcal{P}_{+}(X\times W)\subseteq\mathcal{C}_{+}(X\times W)$. 
\begin{prop}
\label{prop:approximation is tight}The families $\mathcal{P}_{+}(X\times W)$
and $\mathcal{C}_{+}(X\times W)$ are not of type $(3)$.
\end{prop}

\begin{proof}
Let $X=\ints_{\geq1}^{m}$, $W=\{1,...,m\}\subseteq\ints$. Let $p_{1}=2<p_{2}=3<\dots<p_{m+1}$
be the first $m$ prime numbers, and take $f(x_{1},...,x_{m},w)=x_{w}^{p_{w}}$.
Then for any $\epsilon>0$ and any $g\in\mathcal{C}(X)$ with $\underset{1\leq w\leq m}{\sup}f_{F}(x,w)\leq g_{F}(x)$
for $F\in\mathrm{Loc}_{\gg}$, one cannot have 
\[
g_{F}(x)\leq(m-\epsilon)\cdot\underset{1\leq w\leq m}{\sup}f_{F}(x,w)\tag{\ensuremath{\star}}
\]
for each $F\in\mathrm{Loc}_{\gg}$ and $x\in X_{F}$. In fact, $\sum_{j=1}^{m}x_{j}^{p_{j}}$
is an optimal approximation (with constant $m$).

Here is a rough sketch. We assume, towards contradiction, the existence
of $g\in\mathcal{C}(X)$ satisfying $(\star)$. Using Presburger cell
decomposition \cite[Theorem 1]{Clu03}, we can decompose $X$ into
cells $X=\bigsqcup_{i=1}^{N}C_{i}$, such that on each $C_{i}$, the
definable Presburger functions appearing in $g$ are linear. We may
find a large cell of the form 
\[
C=\{(x_{1},...,x_{m})\in\ints_{\geq1}^{m}:x_{j}\geq\alpha_{j}(x_{j+1},...,x_{m})\wedge x_{j}=c_{j}\,\mathrm{mod}\,r_{j}\},
\]
for some linear functions $\alpha_{j}$, and integers $0\leq c_{j}\leq r_{j}$.
The cell $C$ is isomorphic to $\text{\ensuremath{\ints}}_{\geq1}^{m}$
by an affine change of coordinates $\varphi:\text{\ensuremath{\ints}}_{\geq1}^{m}\rightarrow C$,
after which $g_{F}\circ\varphi$ has the form 
\[
\sum_{i=1}^{M}c_{i}(F)\cdot q_{F}^{a_{i1}e_{1}+\dots+a_{im}e_{m}}\cdot\prod_{j=1}^{m}e_{j}^{b_{ij}},\tag{\ensuremath{\star}\ensuremath{\star}}
\]
for $\{(a_{i1},...,a_{im},b_{i1},...,b_{im})\}_{i}$ mutually different
tuples of integers, where $b_{ij}\geq0$. Since $\frac{g_{F}(x_{1},...,x_{m})}{x_{1}^{p_{1}}+...+x_{m}^{p_{m}}}$
is bounded from above and below by constants, it follows that all
$a_{ij}\leq0$. We can therefore write $g$ as: 
\[
g_{F}(x_{1},...,x_{m})=P_{F}(x_{1},...,x_{m})+E_{F}(x_{1},...,x_{m}),
\]
where $P$ is a polynomial with coefficients in $F$ and $E$ consists
of all the terms of $(\star\star)$ with $a_{ij}<0$ for some $j\in\{1,...,m\}$.
Let us write $P_{F}(x_{1},...,x_{m})=\sum_{j=1}^{m}d_{j}x_{j}^{p_{j}}+Q_{F}$
for some polynomial $Q$, which consists of monomials disjoint from
$\{x_{j}^{p_{j}}\}_{j=1}^{m}$, and coefficients $d_{j}$ depending
on $F$. Now the idea is to use the fact that the cell $C$ ``sees''
many asymptotic directions in $\text{\ensuremath{\ints}}_{\geq1}^{m}$,
to deduce: 
\begin{enumerate}
\item For $F\in\mathrm{Loc}_{\gg}$, the coefficients $d_{j}$ are bounded
from below by constants arbitrary close to $1$. 
\item In a certain region $\widetilde{C}$ in $C$, $E_{F}$ is negligible
with respect to $P_{F}$, so that $g_{F}\sim P_{F}$. The assumption
that the $p_{i}$'s are prime numbers, guarantees that in certain
asymptotic directions in $\widetilde{C}$, $Q_{F}$ is negligible
with respect to $\sum_{j=1}^{m}d_{j}x_{j}^{p_{j}}$, so that $g_{F}\sim\sum_{j=1}^{m}d_{j}x_{j}^{p_{j}}$. 
\end{enumerate}
Items $(1)$ and $(2)$ contradict our assumption on $g$. Note that
without the assumption on the $p_{i}$'s, one can get tighter approximations
than $\sum_{j=1}^{m}x_{j}^{p_{j}}$. For example, $\frac{4}{3}\left(x_{1}^{2}-x_{1}x_{2}^{2}+x_{2}^{4}\right)$
gives a tighter upper bound for $\mathrm{max}(x_{1}^{2},x_{2}^{4})$,
than $x_{1}^{2}+x_{2}^{4}$, since $\frac{4}{3}\left(x_{1}^{2}-x_{1}x_{2}^{2}+x_{2}^{4}\right)\leq\frac{4}{3}\mathrm{max}(x_{1}^{2},x_{2}^{4})$. 
\end{proof}
In \cite[Theorem 2.1.3]{CGH18}, an approximation of suprema result
is proven for a more general class $\mathcal{C}^{\mathrm{exp}}(X\times W)$
of motivic exponential functions, which involves additive characters,
and which is furthermore built out of functions which are definable
in the generalized Denef-Pas language. Due to this larger generality,
the approximation shown in \cite[Theorem 2.1.3]{CGH18} is a bit weaker
than type $(1)$ above (in \cite{CGH18}, one approximates $\sup|f|^{2}$
instead of $\sup f$). This is unavoidable, as already seen in Remark
\ref{rem:approximation}. 
\begin{rem}
One can weaken the definition of approximation as follows. For a function
$f\in\mathcal{C}_{+}(X\times W)$, assume there exist motivic functions
$\{g_{i}\}_{i=1}^{m}\in\mathcal{C}_{+}(X)$, with $m\in\nats$ such
that 
\[
\underset{1\leq i\leq m}{\max}\{g_{iF}(x)\}\leq\underset{w\in W_{F}}{\sup}f_{F}(x,w)\leq C\cdot\underset{1\leq i\leq m}{\max}\{g_{iF}(x)\},
\]
where $C$ is as in types $(1)$-$(4)$ above. Using this weaker form
of approximation, we expect $\mathcal{C}_{+}(X\times W)$ to be of
weakened type $(4)$.

One may also weaken $(3)$ by letting the constant $C$ depend on
the number of variables running over $X\times W$, and wonder whether
$\mathcal{C}_{+}(X\times W)$ is of type $(3)$ when weakened in this
sense. 
\end{rem}

\section{\label{sec:Number-theoretic-characterizatio}Number theoretic characterization
of the (FRS) property}

In this section we use Theorem \ref{thm:improved supremum theorem}
to prove a more general form of Theorem \ref{Thm A} for $\mathrm{Loc}_{\gg}$,
providing a full number theoretic characterization of (FRS) morphisms
(Theorem \ref{thm:number theoretic characterization of the (FRS) property}).

Throughout this section, $K$ will be a fixed number field. 

\subsection{\label{subsec:An-analytic-characterization}An analytic characterization
of the (FRS) property}

Given an $\mathcal{O}_{F}$-morphism $\varphi:X\rightarrow Y$, we
denote the natural maps $X(\mathcal{O}_{F}/\frak{\mathfrak{m}}_{F}^{k})\rightarrow Y(\mathcal{O}_{F}/\frak{\mathfrak{m}}_{F}^{k})$
by $\varphi$, therefore $\varphi^{-1}(\overline{y})$ is a finite
set in $X(\mathcal{O}_{F}/\frak{\mathfrak{m}}_{F}^{k})$, for any
$\overline{y}\in Y(\mathcal{O}_{F}/\frak{\mathfrak{m}}_{F}^{k})$.
We denote by $r_{k}:Y(\mathcal{O}_{F})\to Y(\mathcal{O}_{F}/\frak{\mathfrak{m}}_{F}^{k})$
and by $r_{l}^{k}:Y(\mathcal{O}_{F}/\frak{\mathfrak{m}}_{F}^{k})\rightarrow Y(\mathcal{O}_{F}/\frak{\mathfrak{m}}_{F}^{l})$
the natural reduction maps for $k\geq l$. 
\begin{defn}
Let $Y$ be a smooth $F$-variety, with $F\in\mathrm{Loc}$. A measure
$\mu$ on $Y(F)$ is called: 
\begin{enumerate}
\item \emph{Smooth} if for any $y\in Y(F)$ there exists an analytic neighborhood
$U\subseteq Y(F)$ and an analytic diffeomorphism $\psi:U\rightarrow\mathcal{O}_{F}^{\mathrm{dim}Y}$
such that $\psi_{*}\mu$ is a Haar measure on $\mathcal{O}_{F}^{\mathrm{dim}Y}$. 
\item \emph{Schwartz} if it is compactly supported and smooth. 
\end{enumerate}
\end{defn}

\begin{lem}[cf. \cite{Ser81,Wei82,Oes82}]
\label{lem:canonical measure}Let $F$ be in $\mathrm{Loc}$, and
let $Y$ be a finite type $\mathcal{O}_{F}$-scheme such that $Y\times_{\spec\mathcal{O}_{F}}\spec F$
is smooth, of pure dimension $d$. Then there is a unique Schwartz
measure $\mu_{Y(\mathcal{O}_{F})}$ on $Y(\mathcal{O}_{F})$, and
there exists $k_{0}\in\nats$, such that for every $k\geq k_{0}$
and every $\bar{y}\in Y(\mathcal{O}_{F}/\mathfrak{m}_{F}^{k})$, one
has
\begin{equation}
\mu_{Y(\mathcal{O}_{F})}(r_{k}^{-1}(\bar{y}))=q_{F}^{-kd}.\label{eq:property of canonical measure}
\end{equation}
The measure $\mu_{Y(\mathcal{O}_{F})}$ is referred to as the \textbf{canonical
measure} on $Y(\mathcal{O}_{F})$. In the special case when $Y$ is
smooth over $\mathcal{O}_{F}$, then (\ref{eq:property of canonical measure})
holds for every $k\geq1$.
\end{lem}

\begin{proof}
If $Y$ is affine, then the existence and uniqueness of $\mu_{Y(\mathcal{O}_{F})}$
follows from \cite[Lemma 3]{Oes82}, building on \cite[Theorem 9]{Ser81}.
In general, let $Y=\bigcup_{i=1}^{N}U_{i}$ be an open affine cover
by $\mathcal{O}_{F}$-subschemes $U_{i}$. Then $Y(\mathcal{O}_{F})=\bigcup_{i=1}^{N}U_{i}(\mathcal{O}_{F})$.
Note that
\[
\mu_{U_{i}(\mathcal{O}_{F})}|_{U_{i}(\mathcal{O}_{F})\cap U_{j}(\mathcal{O}_{F})}=\mu_{U_{j}(\mathcal{O}_{F})}|_{U_{i}(\mathcal{O}_{F})\cap U_{j}(\mathcal{O}_{F})}
\]
by uniqueness, so we can glue them together to form $\mu_{Y(\mathcal{O}_{F})}$.
If furthermore $Y$ is smooth over $\mathcal{O}_{F}$, then by applying
Hensel's lemma to (\ref{eq:property of canonical measure}) we can
choose $k_{0}=1$ (see also \cite[Theorem 2.25]{Wei82}).
\end{proof}
In \cite{AA16}, Aizenbud and Avni gave an analytic characterization
of the (FRS) property: 
\begin{thm}[{\cite[Theorem 3.4]{AA16}}]
\label{thm:analytic criterion of the (FRS) property}Let $\varphi:X\rightarrow Y$
be a map between smooth $K$-varieties. Then the following are equivalent: 
\begin{enumerate}
\item $\varphi$ is (FRS). 
\item For any $F\in\mathrm{Loc}_{0}$ and any Schwartz measure $\mu$ on
$X(F)$, the measure $\varphi_{*}(\mu)$ has continuous density. 
\item For any $x\in X(\overline{K})$ and any finite extension $K'/K$ with
$x\in X(K')$, there exists $F\in\mathrm{Loc}_{0}$ containing $K'$,
and a non-negative Schwartz measure $\mu$ on $X(F)$ that does not
vanish at $x$ such that $\varphi_{*}(\mu)$ has continuous density. 
\end{enumerate}
\end{thm}

The next result shows the above characterization extends to local
fields of large positive characteristic. 
\begin{cor}
\label{cor:pushforward positive characteristic}Let $\varphi:X\rightarrow Y$
be a map between smooth $K$-varieties. Then $\varphi$ is (FRS) if
and only if for every $F\in\mathrm{Loc}_{\gg}$, the measure $\varphi_{*}(\mu_{X(\mathcal{O}_{F})})$
has bounded density with respect to $\mu_{Y(\mathcal{O}_{F})}$.
\end{cor}

\begin{proof}
Without loss of generality, we may assume that $Y$ is affine. By
choosing an $\mathcal{O}_{K}$-model of $Y$, we may identify it as
an $\mathcal{L}_{\mathrm{DP}}$-definable set. Assume $\varphi$ is
(FRS). For each $F\in\mathrm{Loc}_{\gg}$, write $\tau_{F}:=\varphi_{*}(\mu_{X(\mathcal{O}_{F})})$.
By \cite[Theorem 3.4(2)]{AA16}, we can write $\tau_{F}=f_{F}\cdot\mu_{Y(\mathcal{O}_{F})}$
and $f_{F}$ is continuous, for each $F\in\mathrm{Loc}_{0,\gg}$.
Moreover, locally, $f$ can be written as an integral of a motivic
function $G$ in $\mathcal{C}_{+}(Y\times\VF^{\mathrm{dim}X-\mathrm{dim}Y})$,
over $\VF^{\mathrm{dim}X-\mathrm{dim}Y}$. By \cite[Theorem 4.4.1]{CGH14},
it follows that $G_{F}(y,\cdot)$ is integrable, for each $F\in\mathrm{Loc}_{\gg}$
and $y\in Y(\mathcal{O}_{F})$. By Theorem \ref{integration of motivic},
we can choose $f$ to be in $\mathcal{C}_{+}(Y)$.

By \cite[Appendix B, Theorem 14.6]{ST16} (or more generally, by \cite[Theorem 2.1.2]{CGH18})
and since $Y(\mathcal{O}_{F})$ is compact, there exists $a\in\ints$,
such that for each $F\in\mathrm{Loc}_{0,\gg}$ and each $y\in Y(\mathcal{O}_{F})$,
one has $f_{F}(y)<q_{F}^{a}$. By Theorem \ref{thm:-transfer principle for bounds}
we thus have $f_{F}(y)<q_{F}^{a}$ for each $F\in\mathrm{Loc}_{\gg}$
and each $y\in Y(\mathcal{O}_{F})$, as required. The other direction
follows from Theorem \ref{thm:analytic criterion of the (FRS) property}
combined with \cite[Lemma 3.15]{GH19}, as in the proof of \cite[Proposition 3.16]{GH19}. 
\end{proof}

\subsection{A number-theoretic characterization of the (FRS) Property}

We now recall the Lang-Weil estimates, and set the required notation
to state the main theorem. 
\begin{defn}
~\label{def:constant in Lang weil} 
\begin{enumerate}
\item For a finite type $\mathbb{F}_{q}$-scheme $Z$, we denote by $C_{Z}$
the number of its top-dimensional geometrically irreducible components
which are defined over $\mathbb{F}_{q}$. 
\item Let $\varphi:X\rightarrow Y$ be a morphism between finite type $\ints$-schemes
$X$ and $Y$, and let $y\in Y(\mathbb{F}_{q})$. Then we write $C_{X,q}$:=$C_{X_{\mathbb{F}_{q}}}$
and $C_{\varphi,q,y}:=C_{(X_{\mathbb{F}_{q}})_{y,\varphi}}$. 
\end{enumerate}
\end{defn}

\begin{thm}[The Lang-Weil estimates \cite{LW54}]
\label{thm:Lang-Weil}For every $M\in\nats$, there exists $C(M)>0$,
such that for every prime power $q$, and any finite type $\mathbb{F}_{q}$-scheme
$X$ of complexity at most $M$ (see e.g. \cite[Definition 7.7]{GH19}),
one has 
\[
\left|\frac{\#X(\mathbb{F}_{q})}{q^{\mathrm{dim}X}}-C_{X}\right|<C(M)q^{-\frac{1}{2}}.
\]
\end{thm}

Let $X,Y$ be finite type $\mathcal{O}_{K}$-schemes, with $X_{K},Y_{K}$
smooth and geometrically irreducible, and let $\varphi:X\rightarrow Y$
be a dominant morphism. Let $\mu_{X(\mathcal{O}_{F})}$ and $\mu_{Y(\mathcal{O}_{F})}$
be the canonical measures on $X(\mathcal{O}_{F})$ and $Y(\mathcal{O}_{F})$
for $F\in\mathrm{Loc}$. Since $\varphi$ is dominant, it follows
that $\tau_{F}:=\varphi_{*}(\mu_{X(\mathcal{O}_{F})})$ is absolutely
continuous with respect to $\mu_{Y(\mathcal{O}_{F})}$, and thus has
an $L^{1}$-density (see e.g. \cite[Corollary 3.6]{AA16}), so that
$\tau_{F}=f_{F}(y)\cdot\mu_{Y(\mathcal{O}_{F})}$. When $Y$ is affine,
the collection $f=\{f_{F}:Y(\mathcal{O}_{F})\rightarrow\complex\}_{F\in\mathrm{Loc}_{\gg}}$
can be chosen to be formally non-negative. Indeed, as in the proof
of Corollary \ref{cor:pushforward positive characteristic}, locally,
$f_{F}$ can be written as an integral of a motivic function $G$
in $\mathcal{C}_{+}(Y\times\VF^{\mathrm{dim}X-\mathrm{dim}Y})$, over
$\VF^{\mathrm{dim}X-\mathrm{dim}Y}$. Note there is an open affine
subscheme $U$ of $Y$, such that $\varphi_{K}$ is smooth over $U_{K}$.
Then $G_{F}(y,\cdot)$ is integrable for every $y\in U(F)$ and $F\in\mathrm{Loc}_{\gg}$.
By Theorem \ref{integration of motivic} it follows that $f|_{U}$
is formally non-negative. Since $U(F)$ is dense in $Y(F)$ for $F\in\mathrm{Loc}_{\gg}$,
by extending $f|_{U}$ by $0$ we get a collection of densities on
$\{Y(\mathcal{O}_{F})\}_{F\in\mathrm{Loc}_{\gg}}$ which is formally
non-negative. 

For $F\in\mathrm{Loc}_{\gg}$, define a function $g_{F}$ for $y\in\mathcal{O}_{F}$
and $k\in\ints_{\geq1}$ by 
\[
g_{F}(y,k)=\frac{1}{\mu_{Y(\mathcal{O}_{F})}(B(y,k))}\int_{\widetilde{y}\in B(y,k)}f_{F}(\widetilde{y})\mu_{Y(\mathcal{O}_{F})},
\]
where $B(y,k)=r_{k}^{-1}(r_{k}(y))$. By Theorem \ref{integration of motivic},
it follows that $\{g_{F}:Y(\mathcal{O}_{F})\times\ints_{\geq1}\rightarrow\complex\}_{F\in\mathrm{Loc}_{\gg}}$
is a formally non-negative motivic function.

For every $F\in\mathrm{Loc}_{\gg}$, every $y\in Y(\mathcal{O}_{F})$
and every $k\in\ints_{\geq1}$, we have 
\begin{equation}
g_{F}(y,k)=\frac{\varphi_{*}(\mu_{X(\mathcal{O}_{F})})(B(y,k))}{\mu_{Y(\mathcal{O}_{F})}(B(y,k))}=\frac{\#\varphi^{-1}(r_{k}(y))}{q_{F}^{k(\mathrm{dim}X_{K}-\mathrm{dim}Y_{K})}},\label{eq:(4.2)}
\end{equation}
where the last equality follows from Lemma-Definition \ref{lem:canonical measure},
and the fact that $Y$ is smooth over $\mathcal{O}_{F}$ for $F\in\mathrm{Loc}_{\gg}$.
Set 
\[
h_{F}(y,k)=\frac{\#\left(\varphi^{-1}(r_{k}(y))\cap(r_{1}^{k})^{-1}(X^{\mathrm{sing},\varphi}(k_{F}))\right)}{q_{F}^{k(\mathrm{dim}X_{K}-\mathrm{dim}Y_{K})}}.
\]
The asymptotics of the functions $h$ and $g$, in $q_{F}$ and $k$,
measure how wild are the singularities of $\varphi$. For example,
if $\varphi_{K}$ is smooth, then $h_{F}(y,k)\equiv0$ and $g_{F}(y,k)<C$
for $F\in\mathrm{Loc}_{\gg}$ and some constant $C$. On the other
hand, if $\varphi:\mathbb{A}^{1}\rightarrow\mathbb{A}^{1}$ is the
map $x\mapsto x^{m}$, then $g(0,k)=h(0,k)=q_{F}^{k-\left\lceil \frac{k}{m}\right\rceil }$.

Furthermore, the motivic function $\{h_{F}\}_{F\in\mathrm{Loc}_{\gg}}$
is formally non-negative (Proposition \ref{prop:h is formally non-negative}).
This is used to prove our main theorem, which we state now. 
\begin{thm}
\label{thm:number theoretic characterization of the (FRS) property}Let
$\varphi:X\to Y$ be a dominant morphism between finite type $\mathcal{O}_{K}$-schemes
$X$ and $Y$, with $X_{K},Y_{K}$ smooth and geometrically irreducible.
Then the following are equivalent: 
\begin{enumerate}
\item $\varphi_{K}:X_{K}\to Y_{K}$ is (FRS). 
\item There exists $C_{1}>0$, such that for each $F\in\mathrm{Loc}_{\gg}$,
$k\in\ints_{\geq1}$ and $y'\in Y(\mathcal{O}_{F})$:
\[
h_{F}(y',k)<C_{1}q_{F}^{-1}.
\]
\item There exists $C_{2}>0$ such that for each $F\in\mathrm{Loc}_{\gg}$,
$k\in\ints_{\geq1}$ and $y\in Y(\mathcal{O}_{F}/\frak{\mathfrak{m}}_{F}^{k})$:
\[
\left|\frac{\#\varphi^{-1}(y)}{q_{F}^{k(\mathrm{dim}X_{K}-\mathrm{dim}Y_{K})}}-\frac{\#\varphi^{-1}(r_{1}^{k}(y))}{q_{F}^{\mathrm{dim}X_{K}-\mathrm{dim}Y_{K}}}\right|<C_{2}q_{F}^{-1}.
\]
\item There exists $C_{3}>0$ such that for each $F\in\mathrm{Loc}_{\gg}$,
$k\in\ints_{\geq1}$ and $y\in Y(\mathcal{O}_{F}/\frak{\mathfrak{m}}_{F}^{k})$:
\[
\left|\frac{\#\varphi^{-1}(y)}{q_{F}^{k(\mathrm{dim}X_{K}-\mathrm{dim}Y_{K})}}-C_{\varphi,q_{F},r_{1}^{k}(y)}\right|<C_{3}q_{F}^{-\frac{1}{2}}.
\]
\item There exists $C_{4}>0$ such that for each $F\in\mathrm{Loc}_{\gg}$,
$\varphi_{*}(\mu_{X(\mathcal{O}_{F})})$ has continuous density $f_{F}$
with respect to $\mu_{Y(\mathcal{O}_{F})}$, and for each $y'\in Y(\mathcal{O}_{F})$,
one has: 
\[
\left|f_{F}(y')-C_{\varphi,q_{F},r_{1}(y')}\right|<C_{4}q_{F}^{-\frac{1}{2}}.
\]
\end{enumerate}
\end{thm}

Before we prove Theorem \ref{thm:number theoretic characterization of the (FRS) property},
we first show it implies Theorem \ref{Thm A}. 
\begin{proof}[Proof of Theorem \ref{Thm A}]
We prove $(1)\Rightarrow(3)\Rightarrow(2)\Rightarrow(4)\Rightarrow(1)$.
To prove $(3)\Rightarrow(2)$, we first treat large primes using implication
$(3)\Rightarrow(4)$ of Theorem \ref{thm:number theoretic characterization of the (FRS) property},
and then treat small primes using the Lang-Weil estimates. Implication
$(4)\Rightarrow(1)$ follows from Theorem \ref{thm:analytic criterion of the (FRS) property}. 

Let us assume that Condition $(2)$ holds. By Lemma-Definition \ref{lem:canonical measure},
and by Condition $(2)$, there exists $C_{1}>0$, such that for every
prime $p$, every $y\in Y(\ints/p^{k}\ints)$ and every $k\geq k_{0}$,
one has
\begin{equation}
\frac{\varphi_{*}\mu_{X(\Zp)}(r_{k}^{-1}(y))}{\mu_{Y(\Zp)}(r_{k}^{-1}(y))}=\frac{\mu_{X(\Zp)}(\varphi^{-1}(r_{k}^{-1}(y)))}{p^{-k\mathrm{dim}Y_{\rats}}}=\frac{\mu_{X(\Zp)}(r_{k}^{-1}(\varphi^{-1}(y)))}{p^{-k\mathrm{dim}Y_{\rats}}}=\frac{\#\varphi^{-1}(y)}{p^{k(\mathrm{dim}X_{\rats}-\mathrm{dim}Y_{\rats})}}<C_{1},\label{eq:(2) implies (4)}
\end{equation}
where $\mu_{X(\Zp)}$ and $\mu_{Y(\Zp)}$ are the canonical measures
on $X(\Zp)$ and $Y(\Zp)$. Let $f_{p}$ be the density of $\varphi_{*}\mu_{X(\Zp)}$
with respect to $\mu_{Y(\Zp)}$. Combining (\ref{eq:(2) implies (4)})
with Lebesgue's differentiation theorem, we get for almost all $y'\in Y(\Zp)$:
\[
f_{p}(y')=\lim\limits _{k\to\infty}\frac{\varphi_{*}\mu_{X(\Zp)}(r_{k}^{-1}(r_{k}(y')))}{\mu_{Y(\Zp)}(r_{k}^{-1}(r_{k}(y')))}<C_{1},
\]
which implies Condition $(4)$. 

It is left to prove $(1)\Rightarrow(3)$. The case of large primes
follows from the implication $(1)\Rightarrow(3)$ of Theorem \ref{thm:number theoretic characterization of the (FRS) property}.
It is left to prove $(3)$ for a fixed prime $p$. By Theorem \ref{thm:analytic criterion of the (FRS) property},
we have $f_{p}<C(p)$ for some $C(p)>0$. Using (\ref{eq:(2) implies (4)}),
we deduce that
\begin{equation}
\frac{\#\varphi^{-1}(y)}{p^{k(\mathrm{dim}X_{\rats}-\mathrm{dim}Y_{\rats})}}<C(p),\label{eq:bound for small primes}
\end{equation}
for every $k\geq k_{0}$ and $y\in Y(\ints/p^{k}\ints)$. For $y\in Y(\ints/p^{k}\ints)$
with $k<k_{0}$ we can take the trivial bound $\#\varphi^{-1}(y)\leq\sum_{l=1}^{k_{0}}\#X(\ints/p^{l}\ints)$
to deduce (\ref{eq:bound for small primes}) for every $k\in\nats$.
Using the triangle inequality, and by applying the trivial upper bound
$\#\varphi^{-1}(\bar{y})<\#X(\mathbb{F}_{p})$ for $\bar{y}\in Y(\mathbb{F}_{p})$,
we deduce $(3)$. 
\end{proof}
\begin{rem}
One can easily adapt the proof of Theorem \ref{Thm A} above to prove
a more general statement where the collection $\left\{ \Qp\right\} _{p}$
is replaced with all completions $K_{\mathfrak{p}}$ of a fixed number
field $K$. On the other hand, Theorem \ref{thm:number theoretic characterization of the (FRS) property}
is definitely not true for all $F\in\mathrm{Loc}$ (e.g. take $\varphi(x)=3x$,
and consider unramified extensions of $\rats_{3}$). 
\end{rem}

We now move to the proof of Theorem \ref{thm:number theoretic characterization of the (FRS) property}.
We start with the easier implications, and deal with the more challenging
implication $(1)\Rightarrow(2)$ in Subsection \ref{subsec:harder implication}.
\begin{proof}[Proof of $(2)\Rightarrow(3)\Rightarrow(4)\Rightarrow(5)\Rightarrow(1)$
of Theorem \ref{thm:number theoretic characterization of the (FRS) property}]
\textbf{Implication $(2)\Rightarrow(3)$: }assume that $h_{F}(y',k)<C_{1}q_{F}^{-1}$
for each $F\in\mathrm{Loc}_{\gg}$, each $k\in\nats$ and each $y'\in Y(\mathcal{O}_{F})$.
Set $y:=r_{k}(y')$ and note that:
\begin{align*}
\left|\frac{\#\varphi^{-1}(y)}{q_{F}^{k(\mathrm{dim}X_{K}-\mathrm{dim}Y_{K})}}-\frac{\#\varphi^{-1}(r_{1}^{k}(y))}{q_{F}^{\mathrm{dim}X_{K}-\mathrm{dim}Y_{K}}}\right| & =\left|\frac{\#\varphi_{|X^{\mathrm{sm},\varphi}}^{-1}(y))}{q_{F}^{k(\mathrm{dim}X_{K}-\mathrm{dim}Y_{K})}}+h_{F}(y',k)-\frac{\#\varphi_{|X^{\mathrm{sm},\varphi}}^{-1}(r_{1}^{k}(y))}{q_{F}^{\mathrm{dim}X_{K}-\mathrm{dim}Y_{K}}}-h_{F}(y',1)\right|\\
 & =\left|h_{F}(y',k)-h_{F}(y',1)\right|\leq2C_{1}q_{F}^{-1}.
\end{align*}
where the second equality follows from Hensel's lemma and the inequality
follows from our assumption on $h$. Since $r_{k}$ is surjective
for $F\in\mathrm{Loc}_{\gg}$, this finishes the proof.

\textbf{Implication $(3)\Rightarrow(4)$: }let us first prove that
$\varphi_{K}$ is flat, assuming Condition $(3)$. It is enough to
show that $\varphi_{\mathbb{F}_{p}}$ is flat for infinitely many
prime numbers $p$. Let $p$ be a prime large enough such that $\mathrm{dim}X_{K}=\mathrm{dim}X_{\mathbb{F}_{p}}$,
$\mathrm{dim}Y_{K}=\mathrm{dim}Y_{\mathbb{F}_{p}}$ and such that
Condition $(3)$ holds for $F=\mathbb{F}_{q}((t))$ for any $q$ which
is a power of $p$. Note there are infinitely many primes $p$ such
that $\mathbb{F}_{p}$ is a residue field of $\mathcal{O}_{K}$ for
some prime of $\mathcal{O}_{K}$. Let $x\in X(\mathbb{F}_{q})$ for
such $q$ and let $\widetilde{x}\in J_{k}(X)(\mathbb{F}_{q})\simeq X(\mathbb{F}_{q}[t]/(t^{k+1}))$
be the image of $x$ under the zero section embedding $X(\mathbb{F}_{q})\hookrightarrow J_{k}(X)(\mathbb{F}_{q})$,
so that $r_{1}^{k}(\varphi(\widetilde{x}))=\varphi(x)$. Then by Condition
$(3)$, we have for any $k\in\nats$: 
\begin{equation}
\left|\frac{\#\varphi^{-1}(\varphi(\widetilde{x}))}{q^{(k+1)(\mathrm{dim}X_{K}-\mathrm{dim}Y_{K})}}-\frac{\#\varphi^{-1}(\varphi(x))}{q^{\mathrm{dim}X_{K}-\mathrm{dim}Y_{K}}}\right|<C_{2}\cdot q^{-1}.\label{eq:(4.3)}
\end{equation}
By choosing $q$ to be a suitable power of $p$ we may assume $C_{\varphi,q,\varphi(x)}\geq1$.
Notice that $\#\varphi^{-1}(\varphi(\widetilde{x}))=\#J_{k}(X_{\varphi(x),\varphi_{\mathbb{F}_{q}}})(\mathbb{F}_{q})$.
Since $\mathrm{dim}J_{k}(X_{\varphi(x),\varphi_{\mathbb{F}_{q}}})\geq(k+1)\cdot\mathrm{dim}X_{\varphi(x),\varphi_{\mathbb{F}_{q}}}$
and since $C_{\varphi,q,\varphi(x)}\geq1$, we have by (\ref{eq:(4.3)})
and by the Lang-Weil estimates that 
\[
\mathrm{dim}X_{\varphi(x),\varphi_{\mathbb{F}_{q}}}=\mathrm{dim}X_{K}-\mathrm{dim}Y_{K}=\mathrm{dim}X_{\mathbb{F}_{q}}-\mathrm{dim}Y_{\mathbb{F}_{q}}.
\]
By miracle flatness, we are done.

To prove Condition $(4)$, by the triangle inequality, it is enough
to find $C'_{3}$ such that for each $F\in\mathrm{Loc}_{\gg}$, $k\in\ints_{\geq1}$
and $y\in Y(\mathcal{O}_{F}/\frak{\mathfrak{m}}_{F}^{k})$: 
\[
\left|\frac{\#\varphi^{-1}(r_{1}^{k}(y))}{q_{F}^{\mathrm{dim}X_{K}-\mathrm{dim}Y_{K}}}-C_{\varphi,q_{F},r_{1}^{k}(y)}\right|<C'_{3}q_{F}^{-\frac{1}{2}}.
\]
This follows from the fact that $\varphi_{K}$ is flat, via a relative
variant of the Lang-Weil estimates (see e.g. \cite[Theorem 8.4]{GHb}).

\textbf{Implications $(4)\Rightarrow(5)$ and $(5)\Rightarrow(1)$:}
let $f_{F}$ be the density of $\varphi_{*}(\mu_{X(\mathcal{O}_{F})})$
with respect to $\mu_{Y(\mathcal{O}_{F})}$. By Lebesgue's differentiation
theorem and Condition $(4)$, for almost every $y'\in Y(\mathcal{O}_{F})$,
we have: 
\begin{align}
\left|f_{F}(y')-C_{\varphi,q_{F},r_{1}(y')}\right| & =\left|\lim\limits _{k\to\infty}\frac{\mu_{X(\mathcal{O}_{F})}(\varphi^{-1}(B(y',k)))}{\mu_{Y(\mathcal{O}_{F})}(B(y',k))}-C_{\varphi,q_{F},r_{1}(y')}\right|\\
 & =\left|\lim\limits _{k\to\infty}\frac{\#\varphi^{-1}(r_{k}(y'))}{q_{F}^{k(\mathrm{dim}X_{K}-\mathrm{dim}Y_{K})}}-C_{\varphi,q_{F},r_{1}(y')}\right|<C_{3}q_{F}^{-\frac{1}{2}}.\label{eq:Lebesgue's differentiation}
\end{align}
This also shows that $f_{F}$ is essentially bounded for $F\in\mathrm{Loc}_{\gg}$.
By Corollary \ref{cor:pushforward positive characteristic} and by
Theorem \ref{thm:analytic criterion of the (FRS) property}, it follows
that $f_{F}$ can be chosen to be continuous, so that (\ref{eq:Lebesgue's differentiation})
holds for all $y'\in Y(\mathcal{O}_{F})$. This implies Condition
$(5)$, which implies Condition $(1)$ using the same Corollary \ref{cor:pushforward positive characteristic}. 
\end{proof}

\subsubsection{\label{subsec:harder implication}Proof of the implication $(1)\Rightarrow(2)$}

In this section we will prove the remaining implication of Theorem
\ref{thm:number theoretic characterization of the (FRS) property},
namely $(1)\Rightarrow(2)$. We first observe the following: 
\begin{prop}
\label{prop:h is formally non-negative}Assume that $Y$ is affine.
Then the motivic function $h$ is formally non-negative. 
\end{prop}

\begin{proof}
We first prove the special case with $X$ affine. Assume that $X\subseteq\mathbb{A}^{m}$
is the zero locus of $g_{1},...,g_{l}\in\mathcal{O}_{K}[x_{1},\dots,x_{m}]$.
Since $X$ and $Y$ are affine, the map $\varphi=(f_{1},\dots,f_{n}):X\rightarrow Y\subseteq\mathbb{A}^{n}$
is a polynomial map, thus with $f_{i}\in\mathcal{O}_{K}[x_{1},\dots,x_{m}]$.
Given $y\in Y(\mathcal{O}_{F})$, set: 
\[
S_{y,k,X}:=\left\{ x\in\mathcal{O}_{F}^{m}:\underset{i,j}{\min}\{\mathrm{val}(g_{i}(x)),\mathrm{val}(f_{j}(x)-y_{j})\}\geq k\right\} .
\]
Now, for any $y\in Y(\mathcal{O}_{F})$, we have 
\begin{align*}
\#\varphi^{-1}(r_{k}(y)) & 
=q_{F}^{km}\int_{\mathcal{O}_{F}^{m}}1_{S_{y,k,X}}\left|dx_{1}\wedge...\wedge dx_{m}\right|.
\end{align*}
Moreover, 
\[
\#\left(\varphi^{-1}(r_{k}(y))\cap(r_{1}^{k})^{-1}(X^{\mathrm{sing},\varphi}(k_{F}))\right)=q_{F}^{km}\int_{\mathcal{O}_{F}^{m}}1_{W_{y,k,X}}\left|dx_{1}\wedge...\wedge dx_{m}\right|,
\]
where 
\[
W_{y,k,X}:=\{x\in S_{y,k,X}:r_{1}(x)\in X^{\mathrm{sing},\varphi}(k_{F})\}.
\]
Since $1_{W_{y,X}}$ is formally non-negative, we get by Theorem \ref{integration of motivic}
that $h$ is formally non-negative as well. Now let $X=\bigcup_{i=1}^{N}U_{i}$
be a cover by smooth open affine subschemes $U_{i}$. For each $i$
and $F\in\mathrm{Loc}_{\gg}$ write $V_{i}:=U_{i}(\mathcal{O}_{F})\backslash\bigcup_{j=1}^{i-1}U_{j}(\mathcal{O}_{F})$
and note that 
\begin{align*}
\#\left(\varphi^{-1}(r_{k}(y))\cap(r_{1}^{k})^{-1}(X^{\mathrm{sing},\varphi}(k_{F}))\right) & =\sum_{i=1}^{N}\#\left((\varphi|_{U_{i}})^{-1}(r_{k}(y))\cap(r_{1}^{k})^{-1}(U_{i}^{\mathrm{sing},\varphi}(k_{F})\cap r_{1}(V_{i}))\right)\\
 & =\sum_{i=1}^{N}q_{F}^{km}\int_{\mathcal{O}_{F}^{m}}1_{W_{y,k,i}}\left|dx_{1}\wedge...\wedge dx_{m}\right|,
\end{align*}
where 
\[
W_{y,k,i}:=\{x\in S_{y,k,U_{i}}:r_{1}(x)\in U_{i}^{\mathrm{sing},\varphi}(k_{F})\cap r_{1}(V_{i})\}.
\]
This finishes the proof of Proposition \ref{prop:h is formally non-negative}. 
\end{proof}
We need one more lemma 
which we state in the generality of $E$-smooth morphisms, and which will further be used in the next section.
\begin{lem}
\label{lem:knoweledge on h(y,k)}Let $E\geq1$ be an integer, let
$\varphi$ be as in Theorem \ref{thm:number theoretic characterization of the (FRS) property}
and assume that $\varphi_{K}:X_{K}\to Y_{K}$ is $E$-smooth. Then,
for each $k\in\nats$ there exists a constant $C(k)>0$ such that
for each $F\in\mathrm{Loc}_{\gg}$, one has 
\[
\underset{y\in Y(\mathcal{O}_{F})}{\sup}h_{F}(y,k)<C(k)\cdot q_{F}^{-E}.
\]
\end{lem}

\begin{proof}
Using Theorems \ref{thm:-transfer principle for bounds} and \ref{thm:improved supremum theorem}
it is enough to prove the lemma for $F$ lying in $\mathrm{Loc}_{+,\gg}$.
By Proposition \ref{prop:useful lemma for jet-flat maps} we have
\begin{equation}
J_{k}(X_{k_{F}}^{\mathrm{sm},\varphi_{k_{F}}})=J_{k}(X_{k_{F}})^{\mathrm{sm},J_{k}(\varphi_{k_{F}})},\label{eq:property of flat jets}
\end{equation}
for $F\in\mathrm{Loc}_{+,\gg}$. Let $Z_{\widetilde{y}}:=J_{k}(X_{k_{F}})_{\widetilde{y},J_{k}(\varphi_{k_{F}})}$
be a non-empty fiber of $J_{k}(\varphi_{k_{F}})$ over $\widetilde{y}\in J_{k}(Y)(k_{F})$.
Since $J_{k}(\varphi_{k_{F}})$ is flat and by (\ref{eq:property of flat jets}),
we have
\begin{equation}
Z_{\widetilde{y}}^{\mathrm{sing}}=Z_{\widetilde{y}}\cap J_{k}(X_{k_{F}})^{\mathrm{sing},J_{k}(\varphi_{k_{F}})}=Z_{\widetilde{y}}\cap(\pi_{0,X_{k_{F}}}^{k})^{-1}(X_{k_{F}}^{\mathrm{sing},\varphi_{k_{F}}}).\label{eq:singular locus of fibers of jets}
\end{equation}
The $E$-smoothness of $\varphi_{K}$ implies that the right hand
side is of codimension at least $E$ in $Z_{\widetilde{y}}$. By the
definition of $h$, by the fact that all fibers of $J_{k}(\varphi_{k_{F}})$
are of bounded complexity (for a fixed $k$) and using the Lang-Weil
estimates, the lemma follows. 
\end{proof}
\begin{proof}[Proof of the implication $(1)\Rightarrow(2)$]
We may assume that $Y$ is affine. Theorem \ref{thm:improved supremum theorem}
and Proposition \ref{prop:h is formally non-negative} imply that
there exist a constant $C_{0}>0$ and a motivic function $H$ in $\mathcal{C}_{+}(\ints_{\geq1})$
such that 
\begin{equation}
\underset{y\in Y(\mathcal{O}_{F})}{\sup}h_{F}(y,k)<H_{F}(k)<C_{0}\cdot\underset{y\in Y(\mathcal{O}_{F})}{\sup}h_{F}(y,k).\label{eq:approximation of h}
\end{equation}
It is thus enough to show that $\underset{k}{\sup}\,H_{F}(k)<C_{1}\cdot q_{F}^{-1}$
for some constant $C_{1}$ which is independent of $F$.

By Corollary \ref{cor:pushforward positive characteristic}, by (\ref{eq:(4.2)})
and since $h_{F}\leq g_{F}$, we deduce that the function $(y,k)\mapsto h_{F}(y,k)$
is bounded for each $F\in\mathrm{Loc}_{\gg}$. By (\ref{eq:approximation of h})
also $k\mapsto H_{F}(k)$ is bounded for each $F\in\mathrm{Loc}_{\gg}$.
As in the proof of Claim \ref{claim2} of Theorem \ref{thm:improved supremum theorem},
it follows that there exist a finite set $L$ of $\ints_{\geq1}$
and a constant $C'_{0}>0$ such that 
\begin{equation}
\underset{k}{\sup}\,H_{F}(k)\le C'_{0}\cdot\sum_{k\in L}H_{F}(k).\label{eq:D.F}
\end{equation}
Using (\ref{eq:approximation of h}), (\ref{eq:D.F}), Lemmas \ref{lem:1-smooth is (FRS)}(1)
and \ref{lem:knoweledge on h(y,k)} and by setting $C_{1}:=C_{0}C'_{0}\cdot\sum_{k\in L}C(k)$,
we obtain 
\[
\underset{k}{\sup}~H_{F}(k)\le C'_{0}\sum_{k\in L}H_{F}(k)\leq C'_{0}C_{0}\sum_{k\in L}\underset{y\in Y(\mathcal{O}_{F})}{\sup}h_{F}(y,k)<C_{1}q_{F}^{-1},
\]
for each $F\in\mathrm{Loc}_{\gg}$.This finishes the proof of $(1)\Rightarrow(2)$. 
\end{proof}

\subsection{\label{subsec:Number-theoretic-estimates-for}Number-theoretic estimates
for $E$-smooth and $\varepsilon$-jet-flat morphisms}

In this subsection we use the improved approximation of suprema (Theorem
\ref{thm:improved supremum theorem}), similarly as in Subsection
\ref{subsec:harder implication}, to provide uniform estimates for
$E$-smooth morphisms and $\varepsilon$-jet flat morphisms, improving
\cite[Theorem 8.18]{GHb}. We start by giving a characterization of
$E$-smooth morphisms. 
\begin{thm}
\label{thm:characterization of E-smooth morphisms}Let $E\geq1$ be
an integer, and let $\varphi:X\to Y$ be a dominant morphism between
finite type $\mathcal{O}_{K}$-schemes $X$ and $Y$, with $X_{K},Y_{K}$
smooth and geometrically irreducible. Then the following are equivalent: 
\begin{enumerate}
\item $\varphi_{K}:X_{K}\to Y_{K}$ is $E$-smooth. 
\item There exists $C_{1}>0$, such that for each $F\in\mathrm{Loc}_{\gg}$,
$k\in\ints_{\geq1}$ and $y'\in Y(\mathcal{O}_{F})$:
\[
h_{F}(y',k)<C_{1}q_{F}^{-E}.
\]
\item There exists $C_{2}>0$ such that for each $F\in\mathrm{Loc}_{\gg}$,
$k\in\ints_{\geq1}$ and $y\in Y(\mathcal{O}_{F}/\frak{\mathfrak{m}}_{F}^{k})$:
\[
\left|\frac{\#\varphi^{-1}(y)}{q_{F}^{k(\mathrm{dim}X_{K}-\mathrm{dim}Y_{K})}}-\frac{\#\varphi^{-1}(r_{1}^{k}(y))}{q_{F}^{\mathrm{dim}X_{K}-\mathrm{dim}Y_{K}}}\right|<C_{2}q_{F}^{-E}.
\]
\end{enumerate}
In particular, when $E=2$, the conditions above are further equivalent
to $\varphi_{K}:X_{K}\to Y_{K}$ being flat with fibers of terminal
singularities (see Lemma \ref{lem:1-smooth is (FRS)}).
\end{thm}

\begin{proof}
The proof of $(1)\Rightarrow(2)$ is identical to the proof of $(1)\Rightarrow(2)$
in Theorem \ref{thm:number theoretic characterization of the (FRS) property},
where the only exception is the inequality $\underset{y\in Y(\mathcal{O}_{F})}{\sup}h_{F}(y,k)<C_{1}q_{F}^{-E}$
for $F\in\mathrm{Loc}_{\gg}$ which is similarly obtained using Lemma
\ref{lem:knoweledge on h(y,k)}. $(2)\Rightarrow(3)$ is similar as
in Theorem \ref{thm:number theoretic characterization of the (FRS) property}. 

$(3)\Rightarrow(1)$: recall that condition $(3)$ implies that $\varphi_{K}$
is jet-flat and that
\[
\left|h_{F}(y',k)-h_{F}(y',1)\right|\leq C_{2}q_{F}^{-E},
\]
 for all $F\in\mathrm{Loc}_{\gg}$, $k\in\ints_{\geq1}$ and $y'\in Y(\mathcal{O}_{F})$.
Write $W_{y'}:=(X_{k_{F}})_{r_{1}(y'),\varphi_{k_{F}}}$. We claim
that $(W_{y'})^{\mathrm{sing}}$ is of codimension at least $E+1$
in $W_{y'}$ for all $F\in\mathrm{Loc}_{\gg}$ and $y'\in Y(\mathcal{O}_{F})$.
Indeed, assume $(W_{y'})^{\mathrm{sing}}$ is of codimension $r$
in $W_{y'}$ with $r\leq E$. Identifying $r_{1}(y')$ with $\widetilde{y}:=s_{1}(r_{1}(y'))=(r_{1}(y'),0)\in J_{1}(Y)(k_{F})$
under the zero section embedding $s_{1}:Y\hookrightarrow J_{1}(Y)$,
and using (\ref{eq:singular locus of fibers of jets}) one has 
\[
\left(J_{1}(X_{k_{F}})_{\widetilde{y},J_{1}(\varphi_{k_{F}})}\right)^{\mathrm{sing}}=J_{1}(W_{y'})\cap(\pi_{0,X_{k_{F}}}^{1})^{-1}(X_{k_{F}}^{\mathrm{sing},\varphi_{k_{F}}})=(\pi_{0,W_{y'}}^{1})^{-1}(W_{y'}^{\mathrm{sing}}).
\]
Since the dimension of the Zariski tangent space of a variety $Z$
at a singular point is larger than $\mathrm{dim}Z$, we have
\begin{align*}
\mathrm{dim}\left(J_{1}(X_{k_{F}})_{\widetilde{y},J_{1}(\varphi_{k_{F}})}\right)^{\mathrm{sing}} & \geq\mathrm{dim}(W_{y'})^{\mathrm{sing}}+\mathrm{dim}X_{K}-\mathrm{dim}Y_{K}+1\\
 & \geq\mathrm{dim}W_{y'}-r+\mathrm{dim}X_{K}-\mathrm{dim}Y_{K}+1\\
 & \geq\mathrm{dim}J_{1}(X_{k_{F}})_{\widetilde{y},J_{1}(\varphi_{k_{F}})}-r+1.
\end{align*}
Hence $\left(J_{1}(X_{k_{F}})_{\widetilde{y},J_{1}(\varphi_{k_{F}})}\right)^{\mathrm{sing}}$
is of codimension at most $r-1$. By replacing $F$ with a finite
extension, and using the Lang-Weil estimates, one can find $C_{3}>0$
such that 
\[
h_{F}(y',1)<C_{3}q_{F}^{-r}\text{ and }h_{F}(y',2)>\frac{1}{2}q_{F}^{-r+1}.
\]
But this contradicts Condition $(3)$. Therefore $h_{F}(y',1)<C_{3}q_{F}^{-(E+1)}$
for all $F\in\mathrm{Loc}_{\gg}$ and $y'\in Y(\mathcal{O}_{F})$.
But then by Condition $(3)$, we deduce that $h_{F}(y',k)<C_{3}q_{F}^{-E}$
which implies that $\varphi_{K}$ is $E$-smooth. 
\end{proof}
Finally, a number-theoretic estimate can be given to $\varepsilon$-jet-flat
morphisms, sharpening the estimate in \cite[Theorem 8.18]{GHb}.
\begin{thm}[{cf. \cite[Theorem 8.18]{GHb}}]
\label{thm:characterization of epsilon-jet flatness}Let $\varphi:X\to Y$
be a dominant morphism between finite type $\mathcal{O}_{K}$-schemes
$X$ and $Y$, with $X_{K},Y_{K}$ smooth and geometrically irreducible
and let $0<\varepsilon\leq1$. Then the following are equivalent: 
\begin{enumerate}
\item $\varphi_{K}:X_{K}\to Y_{K}$ is $\varepsilon$-jet flat. 
\item There exist $C,M>0$ such that for each $F\in\mathrm{Loc}_{\gg}$,
$k\in\ints_{\geq1}$ and $y\in Y(\mathcal{O}_{F}/\frak{\mathfrak{m}}_{F}^{k})$,
one has
\[
\frac{\#\varphi^{-1}(y)}{q_{F}^{k(\mathrm{dim}X_{K}-\mathrm{dim}Y_{K})}}<C\cdot k^{M}q_{F}^{k(1-\varepsilon)\mathrm{dim}Y_{K}}.
\]
\end{enumerate}
In particular, when $\varepsilon=1$ and assuming $\varphi_{K}$ has
normal fibers, the conditions above are further equivalent to $\varphi_{K}$
being flat with fibers of log-canonical singularities (Remark \ref{rem:relating jet flatness}).
\end{thm}

The proof of $(2)\Rightarrow(1)$ of Theorem \ref{thm:characterization of epsilon-jet flatness}  follows from \cite[Theorem 8.18]{GHb}.
In order to prove $(1) \Rightarrow (2)$, we prove an auxiliary lemma.
\begin{lem}
\label{lem: auxiliary lemma}
Let $g\in \mathcal{C}_+(\ints_{\geq 1})$ be a formally non-negative motivic function such that for every $\delta >0$ and $ k \in \ints_{\geq 1}$ we have (varying over $F \in \mathrm{Loc}_{\gg}$)
\[
\lim_{q_F \to \infty} q_F^{-\delta} g_F(k)=0.
\]
Then there exist  $M \in \nats$ and  $C>0$ such that $g_F(k) < Ck^M$ for every $k \in \ints_{\geq 1}$ and field 
$F \in \mathrm{Loc}_{\gg}$.
\end{lem}
\begin{proof}
Since $g$ is formally non-negative, we may write $g_F=\sum \# Y_{F,i} f_{F,i}$ for $f_{F,i} \in \mathcal{P}_+(\ints_{\geq 1})$ formally non-negative and $Y_{F,i} \subseteq \ints_{\geq 1} \times \RF^{r_i}$. 
It is enough to show the claim for a single summand $g_F=\#Y_F f_F$.
Using Presburger cell decomposition 
and the orthogonality of $\RF$ and $\VG$, we have a finite partition $\ints_{\geq 1} = \bigcup A_i$ and we may 
write $g_F(k)|_A=\sum \#Y_F c_i(q_F) q_F^{a_i k} k^{b_i}$ on each cell $A$, where $a_{i}\in\rats$, $b_{i}\in\nats$, $\{(a_{i},b_{i})\}_{i=1}^{N}$
 are mutually different, and $c_{i}(q)$
are rational functions in $q$. 

First assume our cell $A$ is finite, in which case it is  enough to prove the claim for a fixed $k=k_0$.
Using \cite[Main Theorem]{CvdDM92}, we have non-negative constants $d$, $C_1$ and $C_2$ such that 
\begin{flalign*}
\label{formula: definable LW}
\tag{$\dagger$}
\#Y_F < C_2 q_F^d ~\text{ for all }F\in \mathrm{Loc}_{\gg},~\text{ and }~
C_1q_{F}^{d} <  \#Y_F
<
C_2q_{F}^{d}
\text{}
\end{flalign*}
for infinitely many fields $F \in \mathrm{Loc}_{\gg}$ (with infinitely many residual characteristics).

Therefore, for every $\delta >0$ and infinitely many fields $F \in \mathrm{Loc}_{\gg}$ we have 
\[
\tag{$\triangle$}
\label{eq: left side of definable LW}
\lim_{q_F \to \infty} 
q_F^{-\delta}\left(
C_1q_{F}^{d}\right)\sum q_F^{a_ik_0}
c_i(q_F) k_0^{b_i}
\leq \lim_{q_F \to \infty} 
q_F^{-\delta} g_F(k_0)
=0,
\]
and thus $\deg_{q}\left(C_2q^d\sum q^{a_ik_0}
c_i(q) k_0^{b_i}\right) \leq 0$ as a rational function in $q$. The claim now follows 
since there exists $C_3>0$ such that for every $F \in \mathrm{Loc}_{\gg}$ with $q_F$ large enough,
\[
g_F(k_0) =
\# Y_F\sum
c_i(q_F)q_F^{a_i k_0} k_0^{b_i}
< C_2q_F^d\sum q_F^{a_ik_0}
c_i(q_F) k_0^{b_i} <C_3.
\]

Now, assume our  cell $A$ is infinite and 
 set $a=\max \{ a_i\}$.
 Using (\ref{eq: left side of definable LW}) with a general $k$ instead of a fixed $k_0$, 
 we must have $a \leq 0$, as otherwise for every $k$ large enough 
 $R(q)=
C_1q^{d}\sum q^{a_ik}
c_i(q) k^{b_i}$
 is a non-zero rational function in $q$ whose degree is positive, and therefore $\lim\limits_{q_F \to \infty} q_F^{-\delta} R(q_F) \neq 0$ for some $\delta >0$.

Set 
$H_F(k)
=\sum\limits_{i: a_i = 0} 
\# Y_F
c_i(q_F)k^{b_i}$ and 
$E_F(k)
=\sum\limits_{i: a_i <0} 
\# Y_F
c_i(q_F)q_F^{a_i k} k^{b_i}$, then we have
\[
g_F(k) =  H_F(k)+ E_F(k) \leq | H_F(k)|+ |E_F(k)|.
\]
Using (\ref{formula: definable LW}), we may find  a 
 constant $C'$  such that $|E_F(k)|<C'$ for every $k$ large enough and $F \in \mathrm{Loc}_{\gg}$. It is therefore left to take care of $H_F(k)$. We may assume $A = \ints_{\geq 1}$.
 
We prove by induction 
 on the number of summands $N$ 
 that if $H_F=\sum\limits_{i=1}^{N} \# Y_F
c_i(q_F) k^{b_i}$ is a function satisfying $\lim\limits_{q_F \to \infty} q_F^{-\delta} H_F(k)=0$  for every $k$ large enough and $\delta>0$, then
 there exists a constant $C''>0$ such that  
for every  $F\in \mathrm{Loc}_{\gg}$ we have $|\#Y_F c_i(q_F)|<C''$ for all $1 \leq i \leq N$.

For $N=1$ the claim follows by (\ref{formula: definable LW}) as before by showing $|\#Y_F c(q_F)|$ is bounded by a rational function of non-positive $q$-degree. 
To prove the claim for $N>1$, consider the functions
\[
\widetilde{H}_{j,F}(k)
={H}_F(2k)-2^{b_j}{H}_F(k)
=\sum\limits_{i=1}^{N} \# Y_F(2^{b_i}-2^{b_j}) 
c_i(q_F) k^{b_i}.
\]
Each $\widetilde{H}_{j,F}(k)$ has $N-1$ summands and satisfies the induction hypothesis since $H_F(k)$ and $H_F(2k)$ do, and therefore the proof by induction is concluded. Using the triangle inequality, we can now find a bound for $H_F(k)$ as required, proving the lemma.
 \end{proof}
\begin{proof}[Proof of Theorem \ref{thm:characterization of epsilon-jet flatness}]
The proof of $(2)\Rightarrow(1)$ follows from \cite[Theorem 8.18]{GHb}.
It is left to prove $(1)\Rightarrow(2)$. 
By Theorem \ref{thm:improved supremum theorem}, there exist $G\in\mathcal{C}_{+}(\ints_{\geq1})$
and $C'>1$ such that 
\[
\underset{y\in Y(\mathcal{O}_{F})}{\sup}g_{F}(y,k)<G_{F}(k)<C'\underset{y\in Y(\mathcal{O}_{F})}{\sup}g_{F}(y,k).
\]
By \cite[Theorem 8.18]{GHb}
and  Theorem \ref{thm:-transfer principle for bounds}, for each 
$0<\varepsilon'<\varepsilon$ we have
$
g_{F}(y,k)<q_{F}^{k((1-\varepsilon')\mathrm{dim}Y_{K})}
$
for all $F\in\mathrm{Loc}_{\gg}$, $k\in\ints_{\geq1}$ and $y\in Y(\mathcal{O}_{F})$.
Therefore, for every $0<\varepsilon'<\varepsilon$ we have
\begin{equation}
G'_F(k):=G_{F}(k)q_{F}^{-k((1-\varepsilon)\mathrm{dim}Y_{K})}
<C'q_{F}^{-k((1-\varepsilon)\mathrm{dim}Y_{K})}q_{F}^{k((1-\varepsilon')\mathrm{dim}Y_{K})}
=C' q_{F}^{k((\varepsilon-\varepsilon')\mathrm{dim}Y_{K})}
\label{eq:assumption on G_F}
\end{equation}
for every $F$ with residual characteristic large enough (which may depend on $\varepsilon'$). 
For any fixed $k$ and $\delta$, 
choose $\varepsilon'$ such that $k((\varepsilon-\varepsilon')\mathrm{dim}Y_{K})< \delta$.	 Using Lemma \ref{lem: auxiliary lemma} on $G'_{F}(k)$ the claim follows.\footnote{Note we may assume $\varepsilon$ is a rational number by Remark \ref{rem:relating jet flatness}.}
\end{proof}
\begin{rem}
To conclude the paper, we note that a possible deeper understanding
of the estimates in Theorems \ref{Thm A}, \ref{thm:number theoretic characterization of the (FRS) property},
\ref{thm:characterization of E-smooth morphisms} and \ref{thm:characterization of epsilon-jet flatness}
may come from the results on exponential sums in \cite{CMN19} and
may be related to the motivic oscillation index $\mathrm{moi}(\varphi)$
of $\varphi$ \footnote{For the definition in the case that $\varphi:\mathbb{A}^{n}\rightarrow\mathbb{A}^{1}$
is a polynomial, see \cite[Section 3.4]{CMN19}.}. The motivic oscillation index controls the decay
rate of the Fourier transform of $\varphi_{*}(\mu_{\mathcal{O}_{F}^{n}})$
(see \cite[Proposition 3.11]{CMN19}). In the non-(FRS) case, optimal
bounds on the decay rate were given in \cite[Theorem 1.5]{CMN19},
proving a conjecture of Igusa on exponential sums \cite{Igu78}. Here
it can also be shown that $\mathrm{moi}(\varphi)$ controls the explosion
rate of the density of the pushforward measure $\varphi_{*}(\mu_{\mathcal{O}_{F}^{n}})$
near a critical point (see e.g. \cite[Theorem 8.18]{GHb}). The (FRS)
case of Igusa's conjecture is open (see the discussion in \cite[Section 3.4]{CMN19}),
and a potential connection between Theorems \ref{thm:number theoretic characterization of the (FRS) property},
\ref{thm:characterization of E-smooth morphisms} and the $\mathrm{moi}(\varphi)$
could be interesting in that regard. 
\end{rem}

\bibliographystyle{alpha}
\bibliography{bibfile}

\end{document}